\documentclass{ws-m3as}

\usepackage{graphicx}
\usepackage{bm}
\usepackage{amsmath}
\usepackage{amssymb}
\usepackage{epsfig}
\usepackage{psfig}
\usepackage{latexsym}
\usepackage{rotating}
\usepackage{amsbsy}
\usepackage{lineno}
\usepackage{stmaryrd}
\usepackage{xparse}
\usepackage[usenames, dvipsnames]{color}
\usepackage[breaklinks,colorlinks=true,allcolors=blue]{hyperref}

\NewDocumentCommand{\dgal}{sO{}m}{%
  \IfBooleanTF{#1}
    {\dgalext{#3}}
    {\dgalx[#2]{#3}}%
}
\NewDocumentCommand{\dgalext}{m}{%
  \sbox0{%
    \mathsurround=0pt 
    $\left\{\vphantom{#1}\right.\kern-\nulldelimiterspace$%
  }%
  \sbox2{\{}%
  \ifdim\ht0=\ht2
    \{\kern-.45\wd2 \{#1\}\kern-.45\wd2 \}%
  \else
    \left\{\kern-.5\wd0\left\{#1\right\}\kern-.5\wd0\right\}%
  \fi
}

\NewDocumentCommand{\dgalx}{om}{%
  \sbox0{\mathsurround=0pt$#1\{$}%
  \sbox2{\{}%
  \ifdim\ht0=\ht2
    \{\kern-.45\wd2 \{#2\}\kern-.45\wd2 \}%
  \else
    \mathopen{#1\{\kern-.5\wd0 #1\{}
    #2
    \mathclose{#1\}\kern-.5\wd0 #1\}}
  \fi
}

\newcommand{\sigmab}{\mbox{\boldmath$\sigma$}}

\font\msbm=msbm10

\newcommand{\R}{\hbox{{\msbm \char "52}}}

\definecolor{otherblue}{rgb}{0,0.3,0.6}
\def\rbl#1{{\textcolor{black}{#1}}}
\def\rblx#1{{\textcolor{black}{#1}}}

\newcommand{\M}{Q}

\begin{document}

\markboth{A.\ KHAN, C.\  E.\ POWELL and D.\  J.\  SILVESTER}
{Robust  Error Estimation for Nearly Incompressible Elasticity}

%
\catchline{}{}{}{}{}
%

\title{ROBUST ERROR ESTIMATION FOR  \rbl{LOWEST-ORDER} APPROXIMATION OF NEARLY INCOMPRESSIBLE ELASTICITY\footnote{This work was supported 
by  EPSRC grant EP/P013317.}}

\author{ARBAZ KHAN}

\address{School of Mathematics, University of Manchester, UK,\footnote{University of Manchester, Oxford Road, Manchester, UK, M13 9PL.}\\
arbaz.khan@manchester.ac.uk}

\author{CATHERINE E. POWELL}
\address{School of Mathematics, University of Manchester, UK,$^{\dagger}$\\
c.powell@manchester.ac.uk}

\author{DAVID J. SILVESTER}
\address{School of Mathematics, University of Manchester, UK,$^{\dagger}$\\
d.silvester@manchester.ac.uk}

\maketitle

\begin{history}
\received{(Day Month Year)}
\revised{(Day Month Year)}
\comby{(xxxxxxxxxx)}
\end{history}

\begin{abstract} We consider so-called Herrmann and Hydrostatic mixed formulations of classical linear elasticity and analyse the error associated with locally stabilised $\bm{P}_1$--$P_0$ finite element approximation.  First, we prove a stability estimate for the discrete problem and  establish an a priori estimate for the associated energy error.  Second, we  consider a residual-based a posteriori error estimator as well as a local Poisson problem estimator. We establish bounds for the energy error that are independent of the Lam\'{e}  coefficients and prove that the estimators are robust in the incompressible limit.  
A key  issue to be addressed is the requirement for pressure stabilisation. 
Numerical results are presented that validate the theory. The software used is available online. 
\end{abstract}



\section{Introduction}\label{sec11}

Our starting point is the  classical  linear boundary value problem modelling the 
deformation of a homogeneous isotropic elastic body,
\begin{subequations}  \label{os1}
\begin{align}  \label{os1a}
 -\nabla\cdot\sigmab& =\bm{f} \quad  \mbox{in } \Omega &&\hspace*{-24pt}(\mbox{equilibrium of forces}),\\
\bm{u} &= {\bm{g} \quad \mbox{on } \partial \Omega}  &&\hspace*{-24pt}(\mbox{essential boundary condition}),
  \label{os1c}
\end{align}
\end{subequations}
where $\Omega \subset \mathbb{R}^{2}$ is a bounded Lipschitz polygon.
Here,   \rbl{the deformation} is written in terms of the 
stress tensor $\sigmab : \R^2\rightarrow \R^{2\times 2}$ and the body force 
$\bm{f}: \R^2\rightarrow \R^{2}$, where 
 \begin{align*}
\sigmab=2 \mu \bm{\varepsilon}(\bm{u})+\lambda ({\nabla\cdot \bm{u}}){\bm{I}},
\end{align*}
${\bm{I}}$ is the $2\times 2$ identity matrix,   $\bm{\varepsilon} : \R^2\rightarrow \R^{2\times 2}$  is the strain 
tensor, $\bm{u}: \R^2\rightarrow \R^{2}$ is the displacement, and  $\bm{\varepsilon}(\bm{u})=\frac{1}{2}(\nabla \bm{u}+(\nabla \bm{u})^{\top})$. The Lam\'{e} coefficients $\mu$ and $\lambda$ satisfy $0<\mu_1<\mu<\mu_2 <\infty$ 
and $0<\lambda<\infty$ and can be written in 
terms of the Young's modulus $E$ and the Poisson ratio $\nu$ as 
\begin{align*}
\mu=\frac{E}{2(1+\nu)}, \quad \lambda=\frac{E\nu}{(1+\nu)(1-2\nu)}.
\end{align*}
The coefficient $\lambda$ becomes unbounded in the incompressible limit $\nu \to 1/2$, leading to  the 
well-known phenomenon of locking for standard finite element methods. A popular remedy is to introduce 
an additional unknown, \rbl{rewrite} \eqref{os1a}--\eqref{os1c} as a system and then apply an appropriate 
\emph{mixed} finite element method.  

We consider mixed approximation methods that are robust with respect to the Lam\'{e} coefficients which arise
from the {\it Herrmann} or {\it Hydrostatic} formulations \cite{DR,RLH} of (\ref{os1a})--(\ref{os1c}). Introducing $p=-\kappa \nabla \cdot \mathbf{u}$  \rbl{we rewrite}  the problem as
\begin{subequations} \label{os2a}
\begin{align}
 -\nabla\cdot\sigmab& =\bm{f} \quad\mbox{in } \Omega, \\
 \nabla\cdot\bm{u}+\frac{p}{\kappa} &=0\quad\mbox{in } \Omega,\\
 \bm{u}&= {\bm{g}\quad \mbox{on } \partial\Omega},
\end{align}
\end{subequations}
where either $\kappa=\lambda$ (in the Herrmann formulation) or $\kappa=\mu+\lambda$ (in the Hydrostatic formulation in two dimensions). The stress tensor can then be written as 
\begin{align} \label{stress-def}
\sigmab(\bm{u}, p) = \left\{\begin{array}{ll}
2 \mu \bm{\varepsilon}(\bm{u})- p{\bm{I}} & \quad \mbox{ (Herrmann)},\\
2 \mu (\bm{\varepsilon}(\bm{u})-\frac{1}{2} (\nabla \cdot \bm{u}) \bm{I})- p{\bm{I}} & \quad \, \mbox{(Hydrostatic)}.
\end{array}\right.
\end{align}

There is an extensive literature on finite element approximation of elasticity problems; see Boffi et al.\cite{DFM} and Hughes\cite{TJH} for a comprehensive overview and  \rbl{Houston et al.\cite{PDT} and Kouhia \& Stenberg\cite{KS}
for specific details}.
In Ref.~\refcite{KPS}, the authors provide a posteriori error analysis for conforming mixed finite element approximations of the Herrmann formulation using \rbl{stable} rectangular elements. \rbl{A variety of} local problem error estimators for the energy error are considered and proved to be robust \rbl{when} $\nu \to 1/2$.  Those results can be extended to the Hydrostatic formulation 
\rbl{whenever} the chosen finite element spaces satisfy \rbl{minimal conditions, as discussed by Boffi \& Stenberg\cite{DR}}. 
\rbl{In this work, we extend the analysis in  Ref.~\refcite{KPS} to cover the {\it lowest-order} $\bm{P}_{1}$--$P_{0}$ 
approximation defined on  triangular elements. {An important}  issue that will be addressed is the requirement
for pressure stabilisation. While pressure stabilisation of the lowest order mixed methods for
the Stokes equations has been extensively studied (for example by
Dohrmann \& Bochev\cite{CP},  Burman \& Fern{\' a}ndez\cite{BF} and Barrenechea \& Valentin\cite{BV})
the application of stabilised methods to elasticity  equations appears to be  a  new development.}


In  Section~\ref{weak_stuff} we review the weak formulation of \eqref{os2a}. In Section~\ref{Hdivmethsec}, 
we discuss $\bm{P}_{1}$--$P_{0}$ approximation, \rbl{review our} local stabilisation strategy and
\rbl{establish an} a priori error bound. 
 \rbl{The stabilisation strategy that is adopted was developed in Refs.}~\refcite{DD,ks92} \rbl{in the context of
 the Stokes equations}. 
\rbl{The distinctive feature of this contribution is the identification of a suitable  energy norm---which 
removes the requirement to  specify (or ``tune'') a stabilisation parameter.} 
In \rbl{Section~\ref{apost}} we discuss a \rbl{conventional} residual-based a posteriori error estimator  and 
\rbl{we introduce} a local Poisson problem estimator.  Both estimators are shown to 
be {\it robust} in the sense that the material parameters do not appear in the error bounds. 
This \rbl{robustness is significantly} more challenging to achieve than for  the Stokes problem, which only
\rbl{involves a single (viscosity) parameter.} \rbl{Some numerical results that reinforce the theory are 
discussed  in  Section~\ref{Numres}.}
\rbl{In the rest of the paper we will use the symbols $\lesssim$ and $\gtrsim$ to denote bounds that are 
valid up to positive constants  that are independent of the Lam\'{e}  coefficients and the mesh  parameters.}

\section{Weak Formulation}\label{weak_stuff}

Our notation is standard: $H^s(\omega)$ denotes the usual  Sobolev space  with norm $||\cdot||_{s,\omega}$ for $s\ge0$. When  $\omega=\Omega$, we use $||\cdot||_{s}$ instead of
$||\cdot||_{s,\Omega}$ and we denote vector-valued Sobolev spaces by boldface letters
$\bm{H}^{s}(\omega)=\bm{H}^{s}(\omega;\R^2)$. We also define
\begin{align*}
\bm{H}^1_E(\Omega):=\bigl\{\bm{v}\in \bm{H}^1(\Omega)
                    \;\big|\; \bm{v}|_{\partial\Omega}=\bm{g}\bigr\}, \quad
\bm{H}^{1\over 2}(\partial \Omega):=\bigl\{\bm{v} \, |\, \bm{v}=\bm{u}|_{\partial\Omega}, \bm{u}\in \bm{H}^1(\Omega)\bigr\} ,
\end{align*} 
and the  test spaces 
${\bm{V}=\bm{H}_{0}^{1}(\Omega) := \bigl\{\bm{v}\in \bm{H}^1(\Omega) \;\big|\; \bm{v}|_{\partial\Omega}=\bm{0} \bigr\}}$ and
$\M=\rbl{L^{2}(\Omega)}$.

The standard weak formulation of (\ref{os2a})  is: 
 find $(\bm{u},p)\in \bm{H}^1_E\times \M$ such that 
 \begin{subequations} \label{scm11a}
\begin{align}
a(\bm{u},\bm{v})+b(\bm{v},p)&=f(\bm{v})\quad \forall \bm{v}\in \bm{V}, \\
b(\bm{u},q)-c(p,q)&=0\quad\quad\,\forall  q\in  \M, \label{scm11b}
\end{align}
\end{subequations}
where
$$ b(\bm{v},p)=-\int_{\Omega} p \nabla\cdot \bm{v},\quad c(p,q)=\frac{1}{\kappa}\int_{\Omega} pq,\quad  f(\bm{v})=\int_{\Omega}\bm{f}\,\bm{v},$$
and either
\begin{align*}
a(\bm{u},\bm{v})=a_{H}(\bm{u}, \bm{v}) =  2\mu\int_{\Omega}\bm{\varepsilon}(\bm{u}):\bm{\varepsilon}(\bm{v}),
 \end{align*} 
(in the Herrmann formulation) or 
\begin{align*}
a(\bm{u},\bm{v})= a_{S}(\bm{u}, \bm{v}) =  2\mu\left(\int_{\Omega}\bm{\varepsilon}(\bm{u}):\bm{\varepsilon}(\bm{v})-\frac{1}{2}\int_{\Omega}(\nabla \cdot \bm{u})(\nabla \cdot \bm{v}) \right)
 \end{align*} 
(in the Hydrostatic formulation). Note that, where it is necessary to make a distinction, we will use the notation $a_{H}(\cdot, \cdot)$ and $a_{S}(\cdot, \cdot)$, but where a stated result holds for both, we will simply use $a(\cdot, \cdot)$. We assume $\bm{f}\in ( L^{2}(\Omega))^2$ and that $\bm{g}\in \bm{H}^{1\over 2}(\partial\Omega)$ is a polynomial of degree at most one in each component so that no error  is incurred in approximating the essential boundary condition.  As usual, we define 
\begin{align}
\mathcal{B}(\bm{u},p; \bm{v},q)=a(\bm{u},\bm{v})+b(\bm{v},p)+b(\bm{u},q)-c(p,q),
\end{align}
so as to  express (\ref{scm11a}) in the more compact form:
find $(\bm{u},p)\in \bm{H}^1_E\times  \M$ such that 
\begin{align}\label{scm12}
\mathcal{B}(\bm{u},p; \bm{v},q)=f(\bm{v}), \quad \forall (\bm{v},q)\in\bm{V}\times  \M.
\end{align}
Finally, we define the following \textit{energy} norm for the error analysis
\begin{align}\label{energy_def}
|||(\bm{u},p)|||^2&=2\mu\, {|| \nabla \bm{u}||^2_{0}}  +(2\mu)^{-1}||p||^2_{0}+ \kappa^{-1} ||p||^2_{0}.
\end{align}

\rbl{One} can establish the well-posedness of the weak formulation for $\nu \in (0, 1/2)$ by considering \eqref{scm11a} 
or  \eqref{scm12}. We will work with the latter. Note that when $\nu = 1/2$, $c(\cdot, \cdot)$ disappears
 from \eqref{scm11a} and the problem can be analysed as a saddle point problem in the standard way 
(similar to Stokes problems). However, since we impose $\bm{u}=\bm{g}$ on the whole boundary,
the pressure solution is only unique up to a constant in that case. 
\rbl{We start by reviewing some useful results.} 
{For both} formulations, it is is easy to show that 
\begin{align}
 a(\bm{u},\bm{v})\le  {2\mu} \, { ||\nabla \bm{u} ||_0}  \,
  \, { ||\nabla \bm{v} ||_0} \quad \forall \bm{u},\bm{v}\in \bm{V}. \label{abd}
 \end{align}
It is also known that that there exists an (inf-sup) constant $C_{\Omega}>0$  such that
\begin{align}
 \sup_{0\neq\bm{v}\in \bm{V}}
\frac{b(\bm{v},q)}{ ||\nabla\bm{v}||_0 }\ge C_{\Omega}   ||q||_{0}, \quad { \forall q\in {\M}, q\neq {\rm constant},}
\label{binfsup}
\end{align}
 see, for example,  p.\,128 of Ref.~\refcite{HDA}.
\rbl{Next,} in the Herrmann formulation, we  \rbl{know that}
\begin{align}
 a_{H}(\bm{v},\bm{v}) & \ge C_K 2\mu \, { ||\nabla \bm{v} ||_0^2} \quad \forall \bm{v}\in \bm{V},   
 \label{aell} 
  \end{align}
 by Korn's inequality, so that $a_{H}(\cdot, \cdot)$ is coercive on $\bm{V}$. 
 {Similarly,} in the Hydrostatic case\cite{DR} \rbl{we have that}
 \begin{align}
 a_{S}(\bm{v},\bm{v}) & \ge(1/2) 2\mu \, { ||\nabla \bm{v} ||_0^2} \quad \forall \bm{v}\in \bm{V}.  
 \label{aell-B} 
  \end{align}
 
 \rbl{We note in passing that  the coercivity estimate \eqref{aell-B} does not hold} if 
 $\partial \Omega = \partial \Omega_{D} \cup \partial \Omega_{N}$, 
 where $\partial \Omega_{N} \neq \emptyset$ is a portion of the boundary where $\bm{\sigma} \bm{n} =\bm{0}$.
\rbl{(This  case requires a separate treatment, exploiting the fact
 that $a_{S}(\cdot,\cdot)$ is coercive on an appropriate nullspace $\bm{V}_{0} \subset \bm{V}$.)} 

The following stability result ensures well-posedness of (\ref{scm12}).
\begin{lemma}\label{Sinsuplem12}
Let $\rblx{\M_0}:=\left\{ q \in \rblx{L^2(\Omega)},  \int_{\Omega} q = 0 \right\}$.
For any $(\bm{u},p)\in \bm{V}\times \rblx{\M_0}$, 
there exists a pair of functions
$(\bm{v},q)\in \bm{V} \times  \rblx{\M_0}$, with  $|||(\bm{v},q)|||\lesssim |||(\bm{u},p)|||$, satisfying
$$ \mathcal{B}(\bm{u},p; \bm{v},q)\gtrsim |||(\bm{u},p)|||^{{2}}. $$
\end{lemma}
\begin{proof}
For the Herrmann case, the result follows from \eqref{abd}, \eqref{binfsup} and \eqref{aell}; see Lemma 3.3 in Ref.~\refcite{KPS}. In the Hydrostatic case, the same proof can be applied, using \eqref{aell-B} instead of \eqref{aell} (which is the same result with $C_{K}=1/2$). Since the energy norm \eqref{energy_def} is defined with respect to $\kappa$, the constant in the bound $\gtrsim$ is the same (up to the value of $C_{K}$).
 \end{proof}
\begin{remark}
\rblx{To check the uniqueness of the pressure solution for $\nu \in (0, 1/2)$  we}
test \eqref{scm11b} with a constant function $\rblx{q=1}$ and use the divergence theorem. This gives
 \begin{align}
  \frac{1}{\kappa}\int_{\Omega} p  \> &=  -\int_{\Omega}  \nabla\cdot \bm{u}
 = \rblx{-} \int_{\partial \Omega}  \bm{u} \cdot  \bm{n} \, ds ,
\quad\hbox{thus}\quad \int_{\Omega} p \> = -   \kappa \int_{\partial \Omega}  \bm{g} \cdot  \bm{n} \, ds . \label{uniquep}
  \end{align} 
The characterisation \eqref{uniquep} guarantees the uniqueness of the pressure \rblx{satisfying \eqref{scm12}} 
using either of the two formulations.
 \end{remark} 

\section{Stabilised $P_{1}$--$P_{0}$ approximation } \label{Hdivmethsec}

Let $\{\mathcal{T}_{h}\}$ denote a family of shape-regular triangular meshes of $\rbl{\overline{\Omega}}$ into triangles $K$ of diameter $h_K$.  For  each mesh $\mathcal{T}_h$, we let $\mathcal{E}_h$ denote the set of all edges and $h_E$ 
\rbl{denote} the length of an edge $E\in\mathcal{E}_h$. Next, we introduce finite-dimensional subsets $\bm{X}^h_E \subset \bm{H}^1_{E}$, $\bm{X}^h_0 \subset \bm{V}$ and {$\M^h \subset \M$}.  The discrete weak formulation of \eqref{scm11a} is as follows: find $(\bm{u}_h,p_h)\in \bm{X}^h_E\times \M^h$ such that 
 \begin{subequations} \label{FEA11}
\begin{align}
a(\bm{u}_h,\bm{v}_h)+b(\bm{v}_h,p_h)&=f(\bm{v}_h)\quad \forall  \bm{v}_h\in \bm{X}^h_0, \\
b(\bm{u}_h,q_h)-c(p_h,q_h)&=0\quad\quad\;\;\forall  q_h\in \M^h.
\end{align}
\end{subequations}
Specifically, we choose $\bm{X}^{h}_{0}$ to be the space of vector-valued functions that are piecewise linear in each component and globally continuous ($\bm{P}_{1}$), and we choose $\M^{h}$ to be the subset of $\M$ that contains piecewise constant functions ($P_{0}$). The solution space $\bm{X}^h_E$ is obtained from $\bm{X}^h_0$ by construction in the usual way,  by augmenting the basis with additional $\bm{P}_{1}$ functions associated with Dirichlet boundary nodes (where $\bm{g} \neq \bm{0}$). For more details about $\bm{P}_1$--${P}_{0}$ approximation, see  Refs.~\refcite{HDA,DD,DFM,ks92,ns98}.  
\rbl{We note that,  while the simplicity of the low-order scheme is very attractive} from a computational point of view,
stabilisation  \rbl{of the underlying  approximation is essential when} working with values of $\nu$ close to $1/2$. 

 
Given a mesh $\mathcal{T}_{h}$, to define our stabilisation strategy, we first select a macroelement partitioning $\mathcal{M}_h$ which satisfies:
\begin{enumerate}
\item Each macroelement $M \in \mathcal{M}_{h}$ is a connected set of adjoining elements from $\mathcal{T}_h$. 
\item $M_i\cap M_j = \emptyset$ for all $M_i, M_j \in \mathcal{M}_{h}, \, i\neq j$. 
\item For any two neighboring macroelements $M_1$ and $M_2$ with $\int_{M_1\cap M_2}ds\neq 0$, there exists $\bm{v}\in 
\rblx{\bm{X}^h_0}$ such that supp $\bm{v}\subset \overline{M}_1\cup \overline{M}_2 $ 
and $\int_{M_1\cap M_2} \bm{v}\cdot \bm{n}\, ds\neq 0$.
\item $\cup_{M\in \mathcal{M}_{h}} \overline{M}= \rbl{\overline{\Omega}}$.
For each $M \in \mathcal{M}_{h}$, the set of \rbl{interior interelement edges} 
will be denoted by $\Gamma_M$.  That is,
$$\Gamma_M=\{ E\in \mathcal{E}_{h} \setminus \partial \Omega, E\subset M\}.$$
\end{enumerate}
With the above definition,  a locally stabilised version of the discrete weak problem (\ref{FEA11}) is as follows: find $(\bm{u}_h,p_h)\in \bm{X}^h_E\times \M^h$ such that 
 \begin{subequations} \label{LSFEA11}
\begin{align}
a(\bm{u}_h,\bm{v}_h)+b(\bm{v}_h,p_h)&=f(\bm{v}_h)\quad \forall \bm{v}_h\in\bm{X}^h_{0}, \\
b(\bm{u}_h,q_h)-c(p_h,q_h)-\mathcal{C}_{loc}(p_h,q_h)&=0,\quad\quad\;\;\forall  q_h\in \M^h, \label{LSFEA12}
\end{align}
\end{subequations}
where 
\begin{align*}
\mathcal{C}_{loc}(p_h,q_h)=\frac{1}{2\mu}\sum_{M\in\mathcal{M}_h}\sum_{E\in\Gamma_M}h_E
\int_{E}\llbracket p_h \rrbracket \llbracket q_h \rrbracket ds, \qquad p_{h}, q_{h} \in \M^{h},
\end{align*}
and $\llbracket\cdot\rrbracket$ denotes the jump across $E \in \Gamma_{M}$. 
 \begin{remark}
The choice of the stabilisation parameter  $(1/2\mu)$ in the definition of $\mathcal{C}_{loc}(\cdot, \cdot)$
is motivated by the  \rblx{a priori error analysis presented next}.
\end{remark}
The discrete \rblx{pressure}  $p_h\in \M^h$ that solves  \eqref{LSFEA11} is not uniquely defined
in the limiting case $\nu=1/2$. The associated linear algebra system is {\it singular} in this case.\footnote{In
the generation of the computational results with  $\nu=0.49999$ (discussed later in Section~\ref{Numres})   
the near-singular linear algebra systems were solved using $\backslash$  within MATLAB.}
\rblx{Define the constrained pressure approximation space  $\M_0^h=\M^h\cap  \M_0$.} 
We will assume that for any partitioning ${\mathcal{M}}_{h}$, each macroelement $M \in {\mathcal{M}}_{h}$ 
belongs to one of a finite number of possible equivalence classes $\mathcal{E}_{\hat{M}_{1}}, \ldots, \mathcal{E}_{\hat{M}_{N}}$. 
The next result immediately follows from  Lemma 3.1 in Ref.~\refcite{ks92}.  
\begin{lemma}\label{alpha1-result}
  Let $\Pi_h$ be the $L^2$ projection operator {from $\rblx{\M_0^h}$} onto the subspace
\begin{align}
\rbl{\overline Q}^{\,h}=\{ q\in \rblx{\M_0}, \, q|_M\, \mbox{ is \rbl{constant} } \forall M\in \mathcal{M}_h \}.
 \end{align}
 Then, \rbl{there exists} $\alpha_{1}>0$ independent of $h$ and the Lam\'e coefficients \rbl{satisfying}
 \begin{align*}
 \mathcal{C}_{loc}(q,q)\ge\alpha_1\frac{1}{2\mu}||(I-\Pi_h)q||_{0}^2 \quad  \forall q\in \rblx{\M_0^h}.
 \end{align*}
 \end{lemma} 
 The stabilised discrete formulation (\ref{LSFEA11}) can {also}
be written as: find $(\bm{u}_h,p_h)\in \bm{X}^h_E\times \M^h$ such that 
\begin{align}\label{LSFEA13}
\mathcal{B}_S(\bm{u}_h,p_h; \bm{v}_h,q_h)=f(\bm{v}_h), \quad \forall (\bm{v}_h,q_h)\in\bm{X}^h_{0}\times \M^h,
\end{align}
 \rbl{which involves the stabilised bilinear} form
\begin{align*}
\mathcal{B}_{S}(\bm{u}_{h},p_{h}; \bm{v}_{h},q_{h})=a(\bm{u}_{h},\bm{v}_{h})+b(\bm{v}_{h},p_{h})+b(\bm{u}_{h},q_{h})-c(p_{h},q_{h})-\mathcal{C}_{loc}(p_h,q_h).
\end{align*}
We are now ready to prove a stability result for \eqref{LSFEA13}.

 \begin{lemma}\label{dinsuplem11}
For any $(\bm{u},p)\in \bm{X}^h_0\times \rblx{\M_0^h}$, there exists a pair of functions 
$(\bm{v},q)\in \bm{X}^h_0\times \rblx{\M_0^h}$ 
with $|||(\bm{v},q)|||\lesssim |||(\bm{u},p)|||$ satisfying
\begin{align*}
 \mathcal{B}_S(\bm{u},p; \bm{v},q)\gtrsim |||(\bm{u},p)|||^2.
\end{align*}
\end{lemma}
\begin{proof}
\rbl{A consequence of \eqref{binfsup}} \rbl{is that} there exists a constant $\alpha_2$, 
independent of $h$ and the Lam\'e coefficients, and a function $\bm{w}\in \bm{X}^h_0$ satisfying
\begin{align}\label{deq11}
(\Pi_h p, \nabla\cdot \bm{w})=(2\mu)^{-1}||\Pi_h p||_{0}^2, \quad (2\mu)^{1/2}||\nabla\bm{w}||_{0}\le\alpha_2(2\mu)^{-1/2}||\Pi_h p||_{0}.
\end{align}
Since $(\bm{u},p)\in \bm{X}^h_{0}\times \rblx{\M_0^h}$ and $\bm{X}_{0}^{h} \subset \bm{V}$, $\rblx{\M_0^h \subset \M^h}$,  
\rbl{using} the definition of $\mathcal{B}_{S}(\cdot, \cdot)$ gives, 
\begin{align}
\mathcal{B}_S(\bm{u},p;\bm{u},-p) & \ge C_{K} 2\mu||\nabla \bm{u}||^2_0+ \kappa^{-1}||p||^2_0+\mathcal{C}_{loc}(p,p), \label{dinfsup1}
\end{align}
by \eqref{aell} (Herrmann case) or \eqref{aell-B} (Hydrostatic case).  Next, using \eqref{deq11} and \eqref{abd}, for any $\epsilon >0$ we have,  
\begin{align}\label{dinfsup2}
\mathcal{B}_S(\bm{u},p; -\bm{w},0) & = -a(\bm{u},\bm{w})-b(\bm{w}, (I-\Pi_h)p)-b(\bm{w}, \Pi_h p)\nonumber\\
&\ge -(2\mu)^{\frac{1}{2}}||\nabla\bm{u}||_0\;(2\mu)^{\frac{1}{2}}||\nabla\bm{w}||_0\nonumber\\
&\quad-(2\mu)^{-\frac{1}{2}}||(I-\Pi_h)p||_{0}(2\mu)^{\frac{1}{2}}||\nabla\bm{w}||_{0}\ + (2\mu)^{-1}||\Pi_h p||^2_{0} \nonumber\\
&\ge -(2\mu)^{1/2}||\nabla\bm{u}||_0\;\alpha_2(2\mu)^{-1/2}||\Pi_h p||_0\nonumber\\
&\quad-(2\mu)^{-1/2}||(I-\Pi_h)p||_{0}(2\mu)^{-1/2}\alpha_2||\Pi_h p||_{0} +  (2\mu)^{-1}||\Pi_h p||^2_{0},  \nonumber\\
&\ge   -\epsilon(2\mu)||\nabla\bm{u}||^2_0- 2\epsilon^{-1} \alpha_2^{2}(2\mu)^{-1}||\Pi_h p||_0^{2} \nonumber\\
&\quad-\epsilon(2\mu)^{-1}||(I-\Pi_h) p||_0^2 + (2\mu)^{-1}||\Pi_h p||^2_{0}.
\end{align} 


Now we introduce a parameter $\delta$. Using Lemma \ref{alpha1-result},  (\ref{dinfsup1}) and (\ref{dinfsup2}) we have,
\begin{align*}
\mathcal{B}_S(\bm{u},p;\bm{u} -\delta \bm{w},-p)&=\mathcal{B}(\bm{u},p;\bm{u},-p)+\delta\mathcal{B}(\bm{u},p; -\bm{w},0)\\
&\ge C_K 2\mu||\nabla\bm{u}||^2_0 + \kappa^{-1}||p||^2_0 + \mathcal{C}_{loc}(p,p)+\delta (2\mu)^{-1}||\Pi_h p||^2_{0}\nonumber\\
&-2\mu\delta\epsilon||\nabla\bm{u}||^2_0
-  2\delta \epsilon^{-1} \alpha_2^2(2\mu)^{-1}||\Pi_h p||_0^2 - \delta\epsilon(2\mu)^{-1}||(I-\Pi_h) p||_0^2,\\
&\ge (C_K-\delta\epsilon)2\mu||\nabla\bm{u}||^2_0+\kappa^{-1}||p||^2_0+\delta \left(1- 2\alpha_2^2 \epsilon^{-1}\right)(2\mu)^{-1}||\Pi_h p||^2_{0}\nonumber\\
&\quad-\delta\epsilon(2\mu)^{-1}||(I-\Pi_h) p||_0^2+\alpha_1(2\mu)^{-1}||(I-\Pi_h) p||_0^2,\\
&\ge (C_K-\delta\epsilon)2\mu||\nabla\bm{u}||^2_0+\kappa^{-1}||p||^2_0+\delta \left(1-\ 2\alpha_2^2 \epsilon^{-1}\right)(2\mu)^{-1}||\Pi_h p||^2_{0}\nonumber\\
&\quad+(\alpha_1-\delta\epsilon)(2\mu)^{-1}||(I-\Pi_h) p||_0^2.
\end{align*}
Making the specific choices $\epsilon =4\alpha_2^2$ and $\delta=\frac{1}{4\alpha_2^2}\min\{C_K/2,\alpha_1/2 \}$, it follows:
\begin{align}\label{infsup11}
\mathcal{B}_S(\bm{u},p;\bm{u} -\delta \bm{w},-p)&\ge C \left(2\mu||\nabla\bm{u}||^2_0+\frac{1}{\kappa}||p||^2_0 +(2\mu)^{-1}\left(||\Pi_h p||^2_{0}+||(I-\Pi_h) p||_0^2\right)\right),\nonumber\\
&\ge C \left(2\mu||\nabla\bm{u}||^2_0+\frac{1}{\kappa}||p||^2_0+(2\mu)^{-1}||p||^2_{0}\right),
\end{align}
where $C=  \min\left\{1,\frac{C_K}{2},\frac{\alpha_1}{2},\frac{\delta}{2}\right\}$. Hence, the result holds with $\bm{v}=\bm{u}-\delta \bm{w}$ and $q=-p$. Finally, using the definition of $||| \cdot ||| $  and \eqref{deq11} gives
\begin{align}\label{infsup21}
 |||(\bm{v},q)|||^2  & = 2 \mu \| \nabla\left(\bm{u}-\delta \bm{w} \right)\|_{0}^{2} + (2\mu)^{-1} \| p \|_{0}^{2} + \kappa^{-1}\|p\|_{0}^{2}   \nonumber \\
& \le  2(2\mu)||\nabla\bm{u}||_{0}^2+ 2(2\mu)\delta^2||\nabla\bm{w}||_{0}^{2} +\left(\kappa^{-1}+(2\mu)^{-1}\right)||p||^{2}_{0},\nonumber\\
&\le \left(2+\frac{\delta^2\alpha_2^2}{2}\right)|||(\bm{u},p)|||^2.
\end{align}
The constants in (\ref{infsup11}) and  (\ref{infsup21}) are independent of the Lam\'e coefficients. \end{proof}

We \rbl{can} now establish an a priori bound for the energy norm of the error associated with the 
stabilised $\bm{P}_{1}$--$P_{0}$ approximation.

\begin{theorem}\label{apriori-thm} Let $\rblx{(\bm{u},p)\in \bm{H}^1_E\times \M}$ 
be the solution to (\ref{scm11a}) 
and let  $(\bm{u}_h,p_h)\in\bm{X}^h_E\times \M^h$  satisfy \eqref{LSFEA11}.  \rblx{Suppose that
 $\int_{\partial\Omega} \bm{g} \cdot  \bm{n} \, ds =0$  so that $\int_\Omega p = 0 = \int_\Omega p_h$
from \eqref{uniquep}.}
If $\bm{u}\in \bm{H}^2(\Omega)$ and $p\in H^1(\Omega)$, then
 \begin{align}
 |||(\bm{u}-\bm{u}_h,p-p_h)|||\lesssim  h \left((2\mu)^{1/2} |\bm{u} |_{2}+\left((2\mu)^{-1/2}+\kappa^{-1/2}\right) | p |_1\right).
 \end{align}
 \end{theorem}
 \begin{proof}
\rblx{Let $\tilde{\bm{u}} \in \bm{X}^h_E$   represent the piecewise linear interpolant of $\bm{u}\in \bm{H}^1_E$} and  let 
$\tilde{p} \in  \M^h_0$ be the piecewise  constant projection  of $p\in\M$ \rblx{with mean value zero}.
\rblx{Using} the triangle inequality gives
 \begin{align}\label{tri-bound}
  |||(\bm{u}-\bm{u}_h,p-p_h)|||  \, \lesssim \,  |||(\bm{u}-\tilde{\bm{u}},p-\tilde{p})||| + |||(\tilde{\bm{u}}-\bm{u}_{h},\tilde{p}- p_{h})|||,
  \end{align}
  and the interpolation error satisfies 
  \begin{align}\label{interp-bound}
 ||| (\bm{u} - \tilde{\bm{u}}, p - \tilde{p}) |||   \, \lesssim \,   h \left( (2\mu)^{1/2} |\bm{u} |_{2}+\left((2\mu)^{-1/2}+\kappa^{-1/2}\right) | p |_1\right).
   \end{align} 
 Now, for all $(\bm{v},q)\in \bm{X}^{h}_{0}\times \M^h$, using (\ref{scm12}) and (\ref{LSFEA13}), \rbl{gives}
 \begin{align*}
\mathcal{B}_S(\bm{u}_h-\tilde{\bm{u}},p_h-\tilde{p} ; \bm{v},q)& = \mathcal{B}_S(\bm{u}_h,p_h ; \bm{v},q)-\mathcal{B}(\tilde{\bm{u}},\tilde{p} ; \bm{v},q)+\mathcal{C}_{loc}(\tilde{p},q) \\
&= \mathcal{B}(\bm{u},p;\bm{v},q)-\mathcal{B}(\tilde{\bm{u}},\tilde{p} ; \bm{v},q)+\mathcal{C}_{loc}(\tilde{p},q) \nonumber\\
&= \mathcal{B}(\bm{u}-\tilde{\bm{u}},p-\tilde{p} ; \bm{v},q)+\mathcal{C}_{loc}(\tilde{p},q).
\end{align*}
Since $(\bm{u}_{h}-\tilde{\bm{u}}, p_{h}-\tilde{p}) \in \rblx{\bm{X}_{0}^{h} \times \rblx{\M_0^h}}$, 
applying Lemma \ref{dinsuplem11}  in the usual way gives
\begin{align*}
|||(\bm{u}_h-\tilde{\bm{u}},p_h-\tilde{p})||| \lesssim ||| \left(\bm{u}-\tilde{\bm{u}},p-\tilde{p} \right) |||
+ \sup_{\rblx{q\in \M_0^h},(2\mu)^{-1/2}||q||_{0}=1}\mathcal{C}_{loc}(\tilde{p},q).
\end{align*}
If $p\in H^1(\Omega)$, then $\mathcal{C}_{loc}({p},q)=0$ and the Cauchy--Schwarz inequality gives
\begin{align*}
\mathcal{C}_{loc}(\tilde{p},q)=\mathcal{C}_{loc}(\tilde{p}-p,q)\lesssim\Bigg(\sum_{K\in \mathcal{T}_h} \frac{h_K}{{2\mu}}||p-\tilde{p}||_{0,\partial K}^2\Bigg)^{1/2}\Bigg(\sum_{K\in \mathcal{T}_h} \frac{h_K}{{2\mu}}||q||_{0,\partial K}^2\Bigg)^{1/2}.\nonumber
\end{align*}
Following the proof of Theorem 3.1 in Ref.~\refcite{ks92}, it follows that
\begin{align*}
\left(\sum_{K\in \mathcal{T}_h} \frac{h_K}{2\mu}||q||_{0,\partial K}^2\right)^{1/2} \lesssim\frac{1}{\sqrt{2\mu}}||q||_{0}, \quad 
\Bigg(\sum_{K\in \mathcal{T}_h} \frac{h_K}{{2\mu}}||p-\tilde{p}||_{0,\partial K}^2\Bigg)^{1/2} \lesssim \frac{h}{\sqrt{2\mu}}|p|_1,
\end{align*}
and hence
\begin{align}\label{other-bound}
|||(\bm{u}_h-\tilde{\bm{u}},p_h-\tilde{p})||| \lesssim ||| \left(\bm{u}-\tilde{\bm{u}},p-\tilde{p} \right) |||+ \frac{h}{\sqrt{2\mu}}|p|_1.
\end{align}
Combining \eqref{tri-bound} with \eqref{interp-bound} and \eqref{other-bound} gives the final result.
\end{proof}

\section{A posteriori error analysis}\label{apost}
\rbl{Two alternative a posteriori {energy} error estimation strategies  will be discussed here. Both 
estimation strategies are {robust} in the sense that material parameters do not appear in the error bounds. 
The proofs are presented here for completeness---they are a minor  extension of the results
established in Ref.~\refcite{KPS}.}

\subsection{Residual error estimation}\label{rposest}
\rbl{We discuss a residual-based error estimator first.  The definition involves three} {distinct} parameters:
\begin{align}\label{gridparams}
\rho_K=h_K(2\mu)^{-\frac{1}{2}},\quad \rho_E=h_E(2\mu)^{-1},\quad\rho_d={1/( \kappa^{-1} + (2\mu)^{-1})}.
\end{align}
Let $(\bm{u}_{h}, p_{h}) \in\bm{X}^h_E\times \M^h$  satisfy \eqref{LSFEA11} and let $\bm{f}_h$ be the $L^2$-projection of $\bm{f}$ onto the space of piecewise constant \rbl{functions}. For each element $K$ in the finite element mesh $\mathcal{T}_{h}$, we define  the local data oscillation error  $\bm{\Theta}_{K}$ satisfying
\begin{align}
 \label{dataapp1}
 \bm{\Theta}_{K}^2=\rho_{K}^2||\bm{f}-\bm{f}_{h}||^{2}_{0,K},
 \end{align}
and a local error indicator $\eta_{K}$ satisfying
$\eta^2_{K}=\eta^2_{R_K}+\eta^2_{E_K}+\eta_{J_K}^2$, where
\begin{align}\label{components}
 \eta^2_{R_K}=\rho_{K}^2||\bm{R}_K||^2_{0,K}, \quad
\eta^2_{J_K}=\rho_d||R_K||^2_{0,K}\quad
\mbox{and}\quad
\eta^2_{E_K}=\sum_{E\in \rbl{\partial K}}\rho_E||\bm{R}_E||^2_{0,E}.
\end{align}
The two {\it element} residuals associated with \eqref{os2a} are given by
\begin{align} \label{elt_resid}
\bm{R}_K=\bm{f}_{h}\big|_K,
\quad R_K=\rblx{\left\{\nabla\cdot \bm{u}_h+ \frac{p_{h}}{\kappa} \right\} \Big|_K},
\end{align}
and the {\it edge} residual $\bm{R}_{E}$ is associated with the normal stress jump.  That is, 
\begin{align} \label{stress-jump-def}
\bm{R}_E=\left\{\begin{array}{ll}
\frac{1}{2}\llbracket -\bm{\sigma}(\bm{u}_{h}, p_{h}) \bm{n}\rrbracket_E & \quad E\in \mathcal{E}_h\setminus\partial\Omega ,\\
0 & \quad {E \in \partial \Omega},
\end{array}\right. 
\end{align}
where $\bm{\sigma}(\bm{u}_{h}, p_{h})$ is defined via \eqref{stress-def}. Note that since $\bm{u}_{h} \in \bm{P}_{1}$ and $p_{h} \in P_{0}$, $\nabla \cdot \bm{u}_{h}$ and \rbl{$R_K$} are constant on each element, \rbl{as  is}
 the normal stress jump on each edge (in both formulations). Hence, $\eta_{K}$ is straightforward to compute.
\rbl{Finally,} we sum the element contributions to give the \rbl{residual error} estimator and data oscillation error respectively,
 \begin{align}\label{errest1}
 \eta=\left(\sum_{K\in\mathcal{T}_{h}}\eta_{K}^2  \right)^{1/2} \quad \hbox{and} \quad
  \bm{\Theta}=\left(\sum_{K\in\mathcal{T}_{h}}\bm{\Theta}_{K}^2\right)^{1/2}.
 \end{align}
\begin{remark}
{The nonuniqueness of the  pressure solution in the incompressible limit
is not seen by the error estimator ($p_h$ drops out of $R_k$ when $\lambda\to\infty$
and $\bm{R}_E$  measures inter-element {\it jumps} in the pressure).}
\end{remark}
Theorems \ref{realiab} and \ref{efficie} show that $\eta$ is a reliable and efficient estimator for the energy error associated with locally stabilised $\bm{P}_{1}$--$P_{0}$ approximations of \eqref{scm11a}. The following standard result is needed for Theorem \ref{realiab}.
 \begin{lemma}[Cl\'{e}ment interpolation]\label{approxlem11}
Given $\bm{v}\in \bm{V},$ let $\bm{v}_h\in \bm{X}^h_0$ be the quasi-interpolant of $\bm{v}$ defined by averaging\cite{CLA}. For any $K\in\mathcal{T}_h$,
 \begin{align*}
 \rho^{-1}_K||\bm{v}-\bm{v}_h||_{0,K}&\lesssim (2\mu)^{1/2} |\bm{v}|_{1,\omega_K},
 \end{align*}
 {where  $|\cdot |_{1,\omega_K}$ is the  $H^1(\omega_K)$ seminorm. Moreover,}  for all $E\in\partial K$  {we have}
  \begin{align*}
  \rho^{-1/2}_E||\bm{v}-\bm{v}_h||_{0,E}&\lesssim (2\mu)^{1/2} |\bm{v}|_{1,\omega_K},
  \end{align*}
  where $\omega_K$ is the set of triangles sharing at least one vertex with $K$.
 \end{lemma}

\begin{theorem}\label{realiab}
Suppose that  $\rblx{(\bm{u},p) \in \bm{H}^1_E\times \M}$ is the weak solution satisfying 
(\ref{scm11a}) and $(\bm{u}_h,p_h)\in\bm{X}^h_E\times \M^h$ satisfies \eqref{LSFEA11}.   {Suppose further that
 $\int_{\partial\Omega} \bm{g} \cdot  \bm{n} \, ds =0$  so that $\int_\Omega p = 0 = \int_\Omega p_h$
from \eqref{uniquep}.}
Defining  $\eta$ and $\bm{\Theta}$ as in \eqref{errest1}, we have
   \begin{align}
     |||(\bm{u}-\bm{u}_h, p-p_h)|||\lesssim \eta +\Theta.
   \end{align}   
 \end{theorem}
 
\begin{proof} Since $(\bm{u}-\bm{u}_{h}, p-p_{h}) \in \rblx{\bm{V} \times \M_0}$, from Lemma \ref{Sinsuplem12}, we have
\begin{align*}
|||(\bm{u}-\bm{u}_h,p-p_h )|||^{2} \lesssim \mathcal{B}(\bm{u}-\bm{u}_h,p-p_h;\bm{v},q)
\end{align*}
for some $(\bm{v}, q) \in \rblx{\bm{V} \times \M_0}$ with 
$|||(\bm{v},q)|||\lesssim  |||(\bm{u} -\bm{u}_h,{p-p_h})|||$. For this $\bm{v}$, choose $\bm{v}_{h} \in \bm{X}_{0}^{h}$ to be defined as in Lemma \ref{approxlem11}. Then, we have $\mathcal{B}(\bm{u}-\bm{u}_{h}, p-p_{h}, \bm{v}_{h}, 0)=0$ by (\ref{scm11a}a) and (\ref{LSFEA11}a). Hence, since $\nabla \cdot \bm{u}+\kappa^{-1}p=0$ and using (\ref{scm11a}a) again,
\begin{align}\label{rea11}
\mathcal{B}(\bm{u}-\bm{u}_h,p-p_h;\bm{v},q)&=\mathcal{B}(\bm{u}-\bm{u}_h,p-p_h;\bm{v}-\bm{v}_h,q),\nonumber\\
&=(\bm{f},\bm{v}-\bm{v}_h)- a(\bm{u}_h, \bm{v}-\bm{v}_h)
     +  (p_h, \nabla\cdot (\bm{v}-\bm{v}_h))  \nonumber\\
& \quad  - (q, \nabla\cdot \bm{u}) + (q, \nabla\cdot \bm{u}_h) 
  - \kappa^{-1}(q, p) + \kappa^{-1}(q, p_h),\nonumber\\
  & = (\bm{f},\bm{v}-\bm{v}_h)-a(\bm{u}_h, \bm{v}-\bm{v}_h)
     +  (p_h, \nabla\cdot (\bm{v}-\bm{v}_h))  \nonumber\\
& \quad + (q, \nabla\cdot \bm{u}_h + \kappa^{-1} p_{h}),\nonumber\\
&=(\bm{f} - \bm{f}_h,\bm{v}-\bm{v}_h) \,  + \! \sum_{K\in \mathcal{T}_h}\!
\Big\{ \big(\bm{f}_h,(\bm{v} - \bm{v}_h) \big)_{{0,K}} \nonumber\\
&\quad + \! \sum_{E\in \partial K} \!\big \langle \bm{R}_E, \bm{v}-\bm{v}_h \big\rangle_E
 + \big(q, R_{K} \big)_{{0,K}} \Big\} 
\end{align}
where $ {\big\langle} \bm{R}_E, \bm{v}-\bm{v}_h {\big\rangle}_E=\int_E \bm{R}_E\cdot (\bm{v}-\bm{v}_h) $. 
Applying Cauchy--Schwarz to (\ref{rea11}) and then using Lemma \ref{approxlem11} gives
\begin{align}\label{rea12}
 ||| (\bm{u}-\bm{u}_h,p-p_h )|||^{2} &\lesssim \mathcal{B}(\bm{u}-\bm{u}_h,p-p_h;\bm{v},q)
\nonumber \\
& \lesssim  \; {  |||(\bm{v},q )||| \;
\Bigg(\sum_{K\in \mathcal{T}_h}\Big(\eta_{K}^2+\Theta_K^2\Big)\Bigg)^{1\over 2} }.
\end{align}
\end{proof}
 \begin{theorem}\label{efficie}
 Suppose that  $\rblx{(\bm{u},p) \in \bm{H}^1_E\times \M}$ is the weak solution satisfying \eqref{scm11a} and 
 $(\bm{u}_h,p_h)\in\bm{X}^h_E\times \M^h$ satisfies \eqref{LSFEA11}. Defining  $\eta$ and $\bm{\Theta}$ as in \eqref{errest1}, we have
      \begin{align}\label{elowerbd}
     \eta\lesssim\, |||(\bm{u}-\bm{u}_h, p-p_h)||| +\Theta.
   \end{align}   
 \end{theorem}

To establish the bound (\ref{elowerbd}), we need to establish efficiency bounds for each of the component residual terms $\eta_{R_K}^{2}$, $\eta_{J_K}^{2}$ and $\eta_{E_K}^{2}$ defined in \eqref{components}. 
\begin{lemma}\label{efficie12}
Let $K$ be an element of $\mathcal{T}_h$. The local equilibrium residual satisfies
\begin{align*}
\eta^2_{R_K}&\lesssim \Big(2\mu \; |\bm{u}-\bm{u}_h|_{1,K}^2+ (2\mu)^{-1}||p-p_h||_{0,K}^2+\Theta_K^2\Big).
\end{align*}
\end{lemma}
\begin{proof} The proof follows the same lines as that of Lemma $3.5$ in Ref.~\refcite{KPS}, here using $\bm{R}_{K}= \bm{f}_{h} = \bm{f}_{h} + \nabla\cdot \bm{\sigma}(\bm{u}_{h}, p_{h})$ (since $\bm{u}_{h} \in \bm{P}_{1}$ and $p_{h} \in P_{0}$) and noting that $ \left(\bm{f} + \nabla \cdot \bm{\sigma}(\bm{u}, p)\right)|_{K}=0$ for a classical solution $(\bm{u},p)$ (in both the Herrmann and Hydrostatic formulations).  In the Hydrostatic formulation, equation (3.22) in Ref.~\refcite{KPS} has the additional term $\mu (\nabla \cdot(\bm{u} - \bm{u}_{h}), \nabla \cdot \bm{w})_{K}$. Applying the Cauchy--Schwarz inequality to this term as well as the others, leads to the stated result.
\end{proof}

\begin{lemma}\label{efficieR2}
Let $K \in \mathcal{T}_h$. The local mass conservation residual satisfies
\begin{align*}
\eta^2_{J_K}&\lesssim \Big(2\mu \; |\bm{u}-\bm{u}_h|_{1,K}^2+ (2\mu)^{-1}||p-p_h||_{0,K}^2
+ \kappa^{-1} ||p-p_h||_{0,K}^2\Big).
\end{align*}
\end{lemma}
\begin{proof}
Noting that $(\nabla\cdot \bm{u}+ \kappa^{-1} p)|_K=0$ for a classical solution $(\bm{u},p)$, we have
\begin{align*}
\rho_d||\nabla\cdot \bm{u}_{h}+ \kappa^{-1} p_h||_{0,K}^2&=\rho_d||\nabla\cdot (\bm{u}-\bm{u}_{h})+ \kappa^{-1} (p-p_h)||_{0,K}^2\\
&\lesssim \rho_d||\nabla\cdot (\bm{u} -\bm{u}_h ) \, ||_{0,K}^2+ \frac{\rho_d}{\kappa^2} ||( p-p_h )||_{0,K}^2\\
&\lesssim 2\mu \; |\bm{u} -\bm{u}_h |_{1,K}^2+ \kappa^{-1} ||(p-p_h)||_{0,K}^2,
\end{align*}
where the last line follows from the definition  of $\rho_d$ in \eqref{gridparams}.
\end{proof}
\begin{lemma}\label{efficieED1}
Let $K \in \mathcal{T}_h$. The stress jump residual satisfies
\begin{align*}
{\eta_{E_K}^2} \lesssim {\sum_{E\in \partial K} \left(2\mu\, |\bm{u}-\bm{u}_h|_{1,\omega_E}^2
+(2\mu)^{-1}||p-p_h||_{0,\omega_E}^2+\Theta_{\omega_E}^2 \right)},
\end{align*}
where $\Theta_{\omega_E}^2 = \sum_{K\in \omega_E}\Theta_K^2$ is the localised data oscillation term and $\omega_{E}$ is the patch of elements that share the edge $E$.
\end{lemma}
\begin{proof} The proof follows the same lines as that of Lemma $3.7$ in Ref.~\refcite{KPS}, but  with $\bm{R}_{E}$ defined as in \eqref{stress-jump-def}, replacing $\llbracket (p_{h}\bm{I} - 2 \mu\bm{\varepsilon}(\bm{u}_{h}))\bm{n}\rrbracket_E$ with $\llbracket -\bm{\sigma}(\bm{u}_{h}, p_{h}) \bm{n}\rrbracket_E $ 
and choosing $\Lambda=\rho_{E} \llbracket -\bm{\sigma}(\bm{u}_{h}, p_{h}) \bm{n}\rrbracket_{E}  \, \chi_{E}$. We again exploit the fact that the classical solution $(\bm{u}, p)$ satisfies $-\nabla \cdot \bm{\sigma}(\bm{u}, p) = \bm{f}$ and $\nabla \cdot \bm{\sigma}(\bm{u}_{h}, p_{h})=0.$ To obtain the upper bound for $\rho_{E} \| \bm{R}_{E}\|_{0,E}^{2}$ in the proof of Lemma $3.7$ there is an additional term to bound for each of the terms $T_{1}$ and $T_{3}$. However the same upper bounds hold.
\end{proof}
\rbl{The desired local lower bound} 
(\ref{elowerbd}) follows by consolidating the estimates from Lemma~\ref{efficie12}, Lemma~\ref{efficieR2} and Lemma~\ref{efficieED1}.

\subsection{\rbl{A Poisson problem local error estimator}}\label{SLEPE}   

Having established that the residual error estimator $\eta$ in (\ref{errest1}) is reliable and efficient, the framework established by  Verf\"urth\cite{RV} makes it  straightforward to construct equivalent {\it local problem} estimators that are equally reliable but potentially more efficient. For the Herrmann formulation with $\bm{Q}_{2}$ (biquadratic) displacement approximation, four local problem error estimators were discussed in Ref. \refcite{KPS}. Of these, the so-called Poisson problem estimator \rbl{proved to be} the most  attractive
\rbl{from a computational perspective.} 

\rbl{This strategy will be extended to cover   stabilised $\bm{P}_{1}$--$P_{0}$ approximation herein.}  
\rbl{We compute} a local estimator $\bm{e}_{P,K} \in ({\mathcal P}(K))^{2}$ for the displacement error 
that is \rbl{{\it super-quadratic}} in each component and a local estimator $\epsilon_{S,K} \in P_{1}(K)$ for the pressure error
that is linear. 
\rbl{More specifically,} for the displacement error, we define 
\begin{align}
\mathcal{P}(K) = \textrm{span}\{ \psi_{E}, \, E \in \partial K \cap \left( {\mathcal{E}}_{h} \setminus \partial \Omega\right) \}
\rbl{\>\oplus\> B_T}, 
\end{align}
where $\psi_{E}$ is a quadratic bubble function associated with an interior edge $E$ 
\rbl{and $B_T$ is the space spanned by the cubic bubble function that is zero on the three boundary edges.}
\rbl{We assume} that every triangle $K \in {\mathcal T}_{h}$ has at least two edges in the interior of $\Omega$.  
\rbl{See} Kay \& Silvester\cite{DD} (and references therein) where the same error estimation strategy 
\rbl{is applied to} Stokes problems. 
The Poisson problem estimator is now defined by
$$\eta_{P} = \sqrt{\sum_{K \in {\cal T}_{h}} \eta_{P,K}^{2}},$$
where the local contributions are given by
\begin{align}\label{poiloc1}
\eta_{P,K}^2= 2\mu \, ||\nabla \bm{e}_{P,K}||^2_{0,K} +\rho_d^{-1}||\epsilon_{P,K}||^2_{0,K},
\end{align}
and $(\bm{e}_{P,K},\epsilon_{S,K})\in (\mathcal{P}(K))^{2} \times {P}_1(K)$ is the solution to the following problem
\begin{subequations}  \label{poiloc}
\begin{align}\label{poiloc2}
2\mu \, (\nabla\bm{e}_{P,K},\nabla\bm{v})_K &=(\bm{R}_K,\bm{v})_K
-\sum_{E\in\partial K} \langle \bm{R}_E, \bm{v} \rangle_E, \quad \forall \bm{v}\in (\mathcal{P}(K))^{2}, \\
\rho_d^{-1}(\epsilon_{P,K}, q)_K &=(R_K,q)_K, \quad \forall q\in P_1(K).\label{poiloc3}
\end{align}
\end{subequations}  
Recall that $\rho_{d}$ is defined in \eqref{gridparams}, $\bm{R}_{K}$ and $R_{K}$ are defined in \eqref{elt_resid} and $\bm{R}_{E}$ is defined in \eqref{stress-jump-def}. With the exception of $\bm{R}_{K}$, these quantities are slightly different depending on which mixed formulation is used. In both cases, (\ref{poiloc2}) decouples into a
pair of local \rbl{$4 \times 4$} Poisson problems and since $R_{K} \in P_1(K)$, the
solution of (\ref{poiloc3}) is immediate: $\epsilon_{P,K} = \rho_d R_K = \rho_d (\nabla\cdot \bm{u}_h + \kappa^{-1} p_h)$. Hence, (\ref{poiloc1}) simplifies to 
$$\eta_{P,K}^2 =2\mu \, ||\nabla \bm{e}_{P,K}||^2_{0,K} + \rho_d || \nabla\cdot \bm{u}_h + \kappa^{-1} p_h ||^2_{0,K}.$$
We note that this strategy of decoupling the components of local problem error estimators in a mixed setting it not new; it was pioneered  by Ainsworth \& Oden\cite{MJ}. 
\rbl{Using the arguments that are} sketched in Ref.~\refcite{KPS}, \rbl{the equivalence result}
$$\eta_{P,K} \lesssim \eta_{K} \lesssim \eta_{P,K}, \qquad K \in {\mathcal T}_{h},$$
\rbl{is easily established.} 


\section{Computational results}\label{Numres}

\rbl{In this section we} 
compare the performance of the estimators $\eta$ and $\eta_P$ for the Herrmann and Hydrostatic formulations of three test problems. 
\rbl{All results} were computed using locally stabilised $\bm{P}_{1}$--$P_{0}$ approximation with software adapted from the MATLAB 
toolbox TIFISS\cite{TIFISS}. To define the stabilisation term, we group the elements in the meshes into disjoint macroelements 
consisting of four neighbouring triangles, {with a central element connected to three neighbours\cite{ks92}}.  
In some experiments we use uniform meshes and in others we use the local contributions 
$\eta_{K}$ and $\eta_{P,K}$ to drive \emph{adaptive} mesh refinement. 
More precisely, starting with an initial mesh $\mathcal{T}_{0},$ we apply the iterative refinement loop
\begin{align*}
\mbox{Solve}\rightarrow\mbox{Estimate}\rightarrow\mbox{Mark}\rightarrow\mbox{Refine}
\end{align*} 
to generate a sequence of (nested) regular meshes $\{\mathcal{T_\ell}\}$ with mesh size $h_{\ell}$. For each $\mathcal{T_\ell}$ and the associated finite element approximation, we compute $\eta_{\ell}^{2}=\sum_{K\in {\mathcal{T}}_{\ell}} \eta_{K}^{2}$ (if using the residual estimator), or else replace $\eta_{K}$ with $\eta_{P,K}$ (if using the Poisson estimator). Then, in the usual way\cite{Doerfler}, using a bulk parameter $\theta\in(0,1)$ (here $\theta=1/2$), we determine a minimal subset $\mathcal{M}_\ell$ of marked triangles such that $\sum_{K\in {\mathcal{M}}_{\ell}} \eta_{K}^{2} \ge \theta\eta_\ell^2$ (and similarly with $\eta_{P,K}$).  Mesh refinement is then done using the red-green-blue strategy\cite{RV}. We denote the number of degrees of freedom associated with the mesh $\mathcal{T}_{\ell}$ by $N_{\ell}.$ Hence, for uniform meshes we have $\mathcal{O} (N^{-r}_{\ell})\approx \mathcal{O}(h^{2r}_{\ell})$ where $r>0$.
{From Theorem \ref{apriori-thm} we know that, if the solution} $(\bm{u}, p)$ is sufficiently smooth, \rbl{then the 
energy error  $e=|||(\bm{u}-\bm{u}_{h}, p-p_{h}) |||$ will decay to zero with} rate $r=0.5$. 

\subsection{\rbl{An analytic solution}}
The first test problem is taken from Ref.~\refcite{CJ}. We choose $\Omega=(0,1)\times (0,1)$ and a zero essential boundary 
condition; \rbl{that is,}  $\bm{g}=\bm{0}$ on $\partial \Omega$. In addition, 
\begin{align*}
\bm{f}= \left(\begin{array}{c} -2\mu\pi^3\cos(\pi y)\sin(\pi y)(2\cos(2\pi x) -1) \\ 2\mu\pi^3\cos(\pi x)\sin(\pi x)(2\cos(2\pi y)-1)
\end{array} \right).
\end{align*}
The exact solution is $p=0$ and $\textbf{u}=(u_1,u_2)^{\top}$
where
\begin{align*}
u_1&=\pi\cos(\pi y)\sin^2(\pi x)\sin(\pi y),\quad
u_2= -\pi\cos(\pi x)\sin^2(\pi y)\sin(\pi x).
\end{align*}
Figures~\ref{Ex1mu100pes} and \ref{Ex1mu100res} show the convergence behaviour of the exact error $e$  as well as the estimated errors obtained with $\eta_{P}$ and $\eta$, respectively, using adaptively generated meshes.  
(The initial mesh ${\mathcal{T}}_{0}$ was \rbl{generated} with $N_{0}=1,090$ degrees of freedom.) 
Here, $\mu$ is fixed and we consider two values of the Poisson ratio $\nu$. The estimated errors converge to zero at the optimal rate ($r=0.5$). \rbl{While both estimators are obviously efficient and reliable for either formulation, the results in Figure \ref{Ex1mu100pes} 
show that the Poisson estimator is the more accurate of the two---the effectivity indices} for the Poisson estimator 
are close to unity even when $\nu \to 1/2$. 
 Identical results (not reported) were obtained when the experiments were repeated with $\mu=1$ and $\mu=0.01$.  We conclude that both \rbl{estimation strategies} are robust with respect to variations in  the parameters  $\mu$ and $\nu$.

\begin{figure}[th!]
\centering
\includegraphics[width=.41\textwidth]{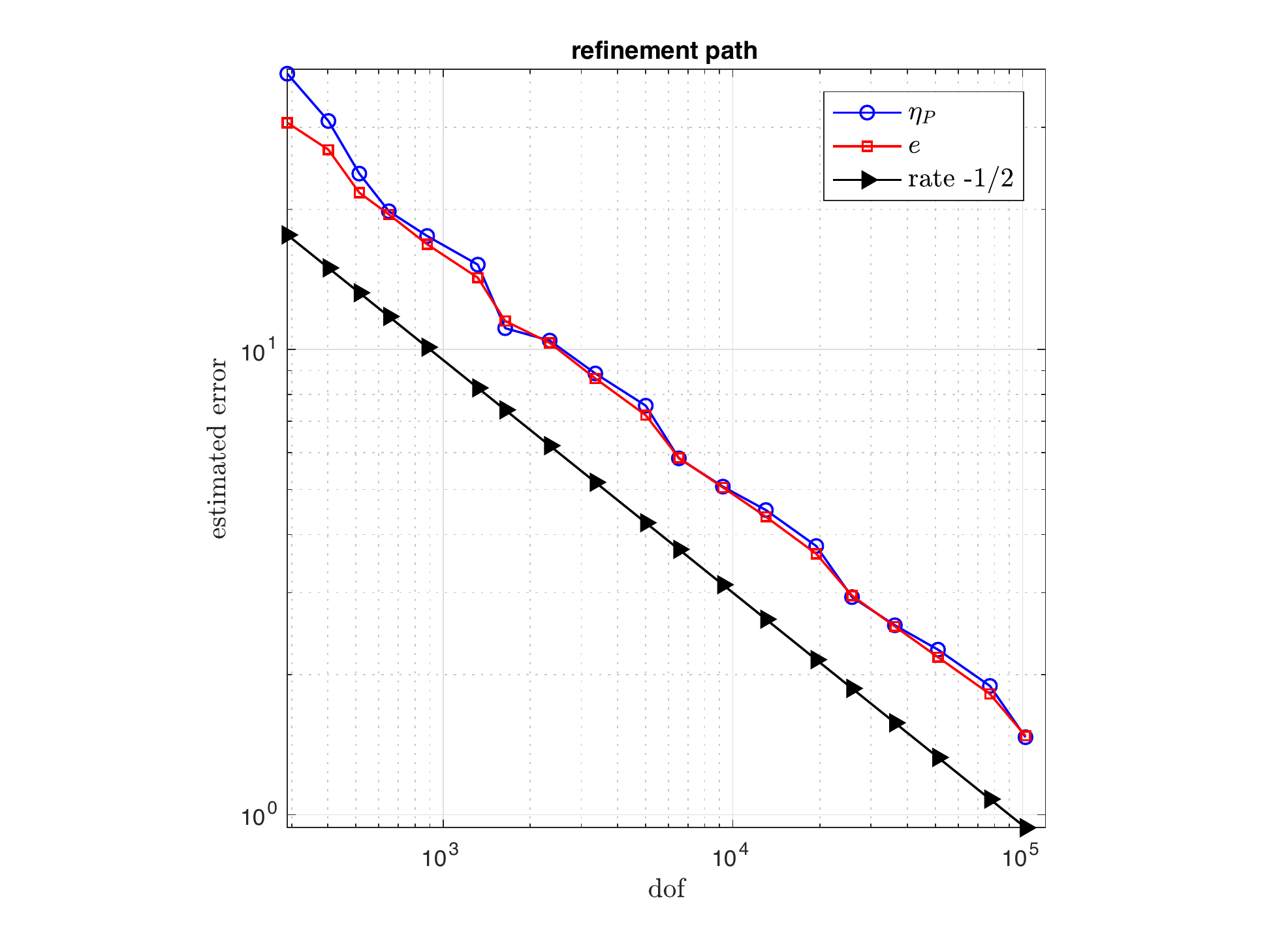}
\label{ex1mu100nu4}
\centering
\includegraphics[width=.41\textwidth]{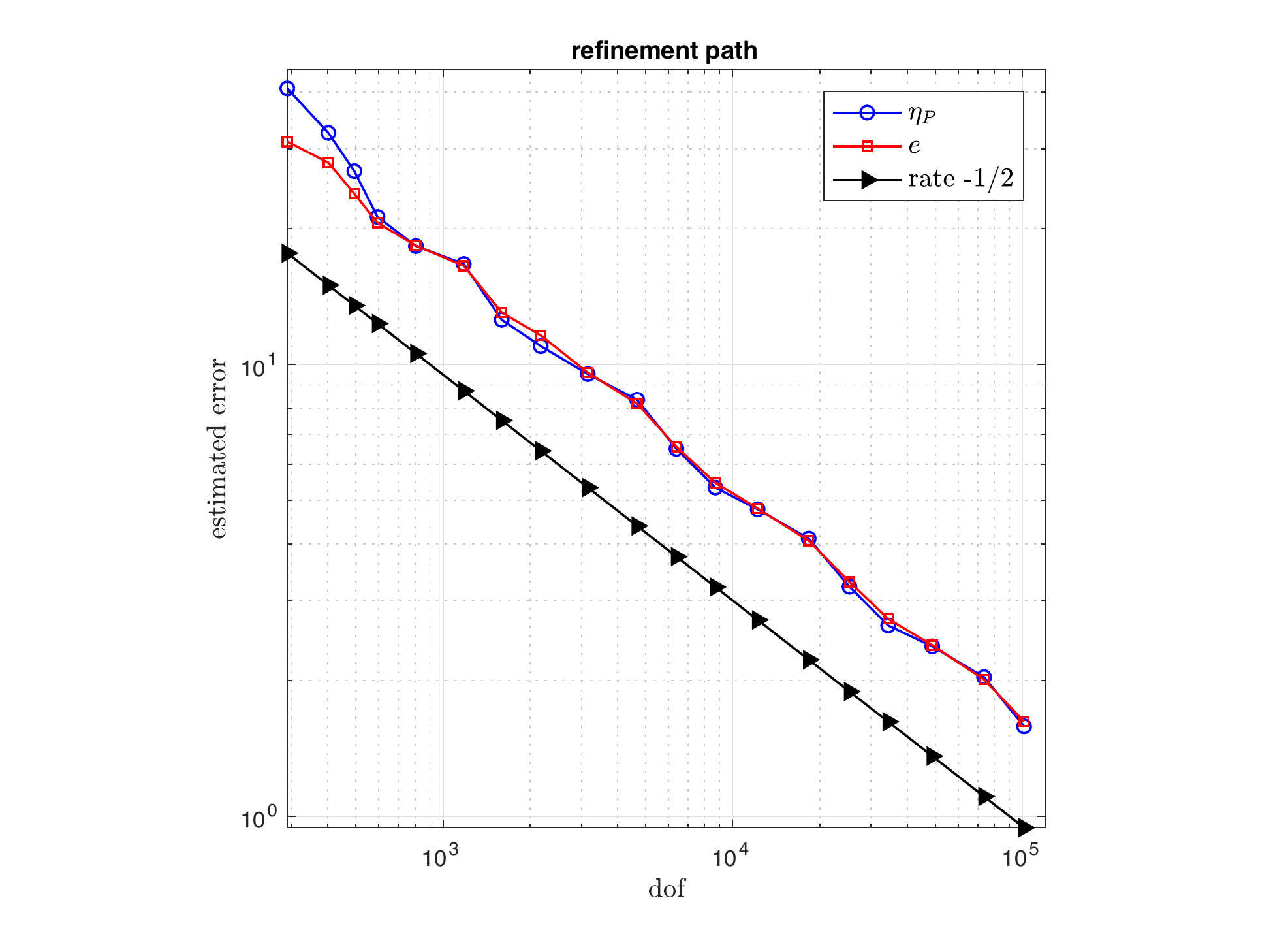}
\label{ex1mu100nu49999}
\centering
\includegraphics[width=.41\textwidth]{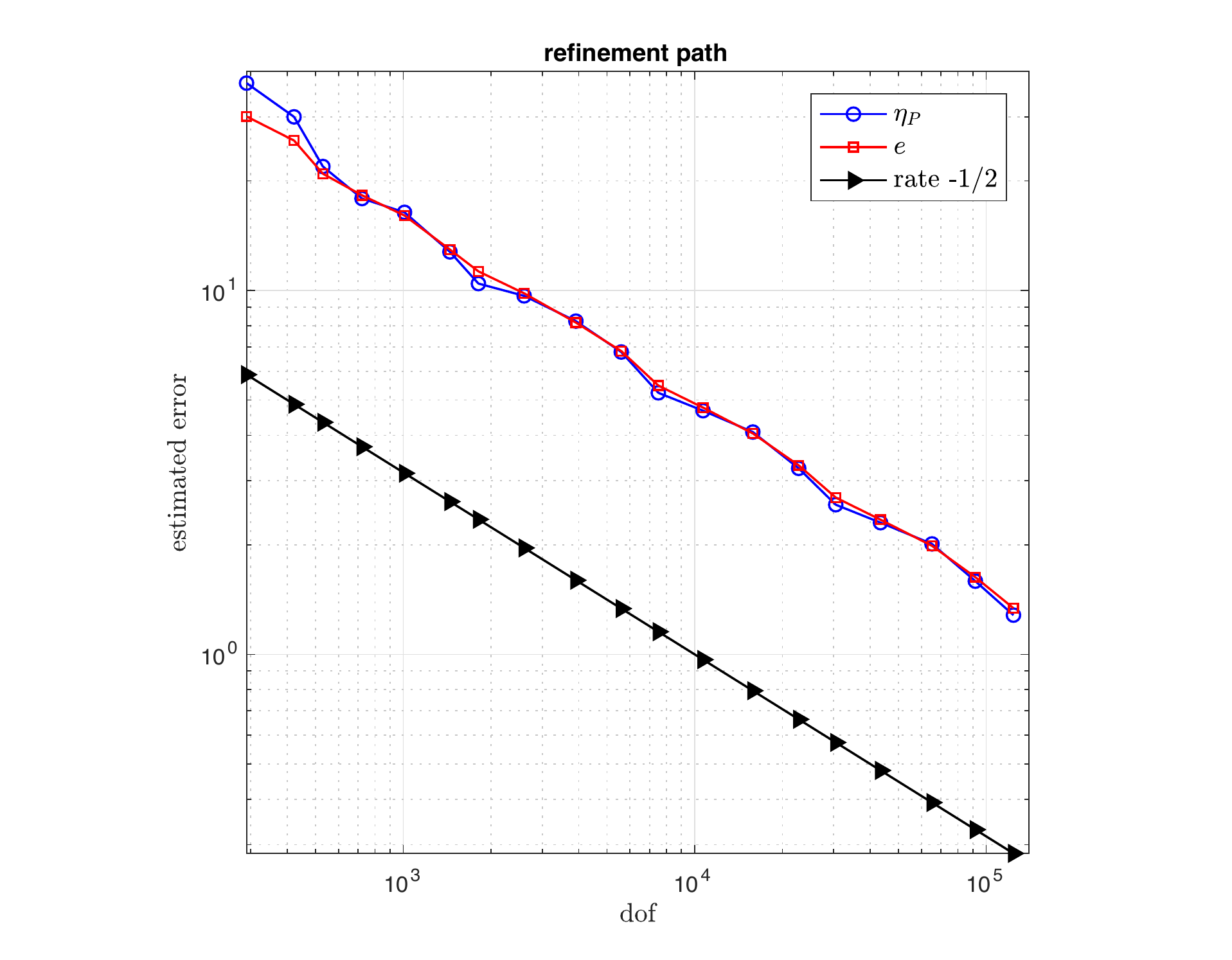}
\label{ex1mu100nu4hydro}
\centering
\includegraphics[width=.41\textwidth]{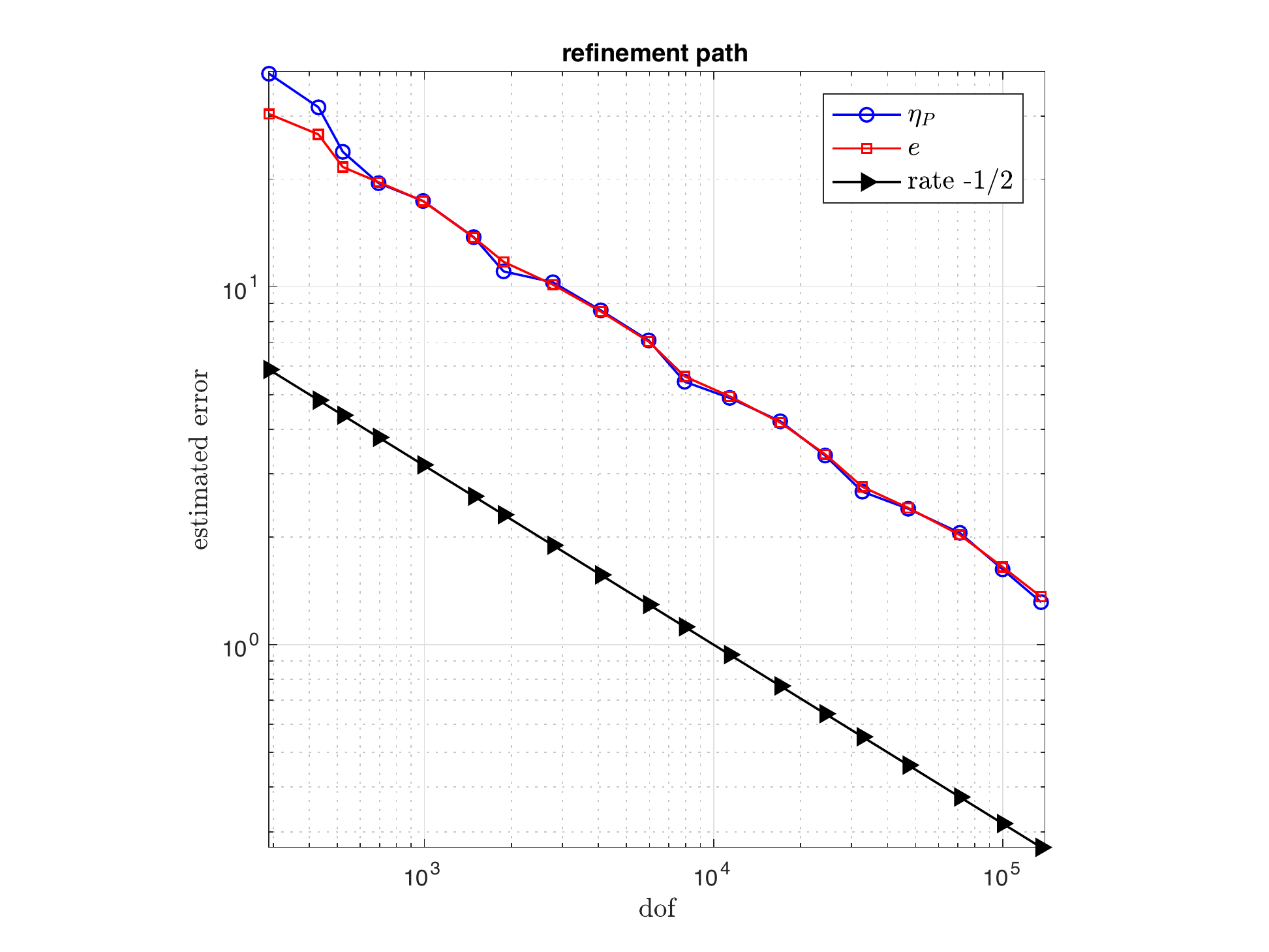}
\label{ex1mu100nu49999hydro}
\caption
{Exact ($e$) and estimated (using the local Poisson estimator $\eta_{P}$) energy errors computed using adaptive meshes, \rbl{for}  Herrmann (top) and Hydrostatic (bottom) formulations of test problem 1, 
with $\mu=100$  and $\nu=0.4$ (left);    $\mu=100$  and $\nu=0.49999$  (right).}
\label{Ex1mu100pes}

\includegraphics[width=.38\textwidth]{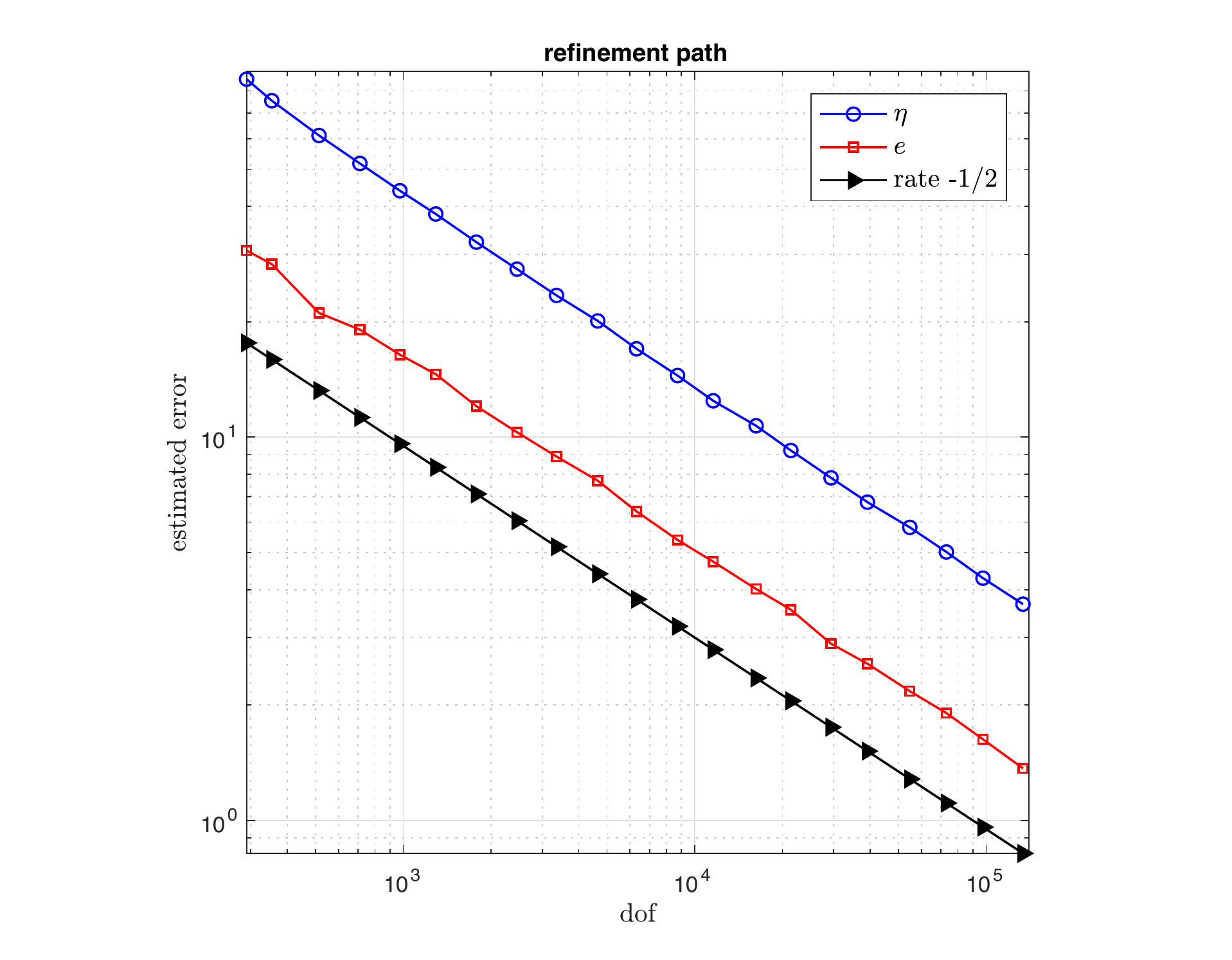}
\label{ex1mu100nu4res}
\centering
\includegraphics[width=.41\textwidth]{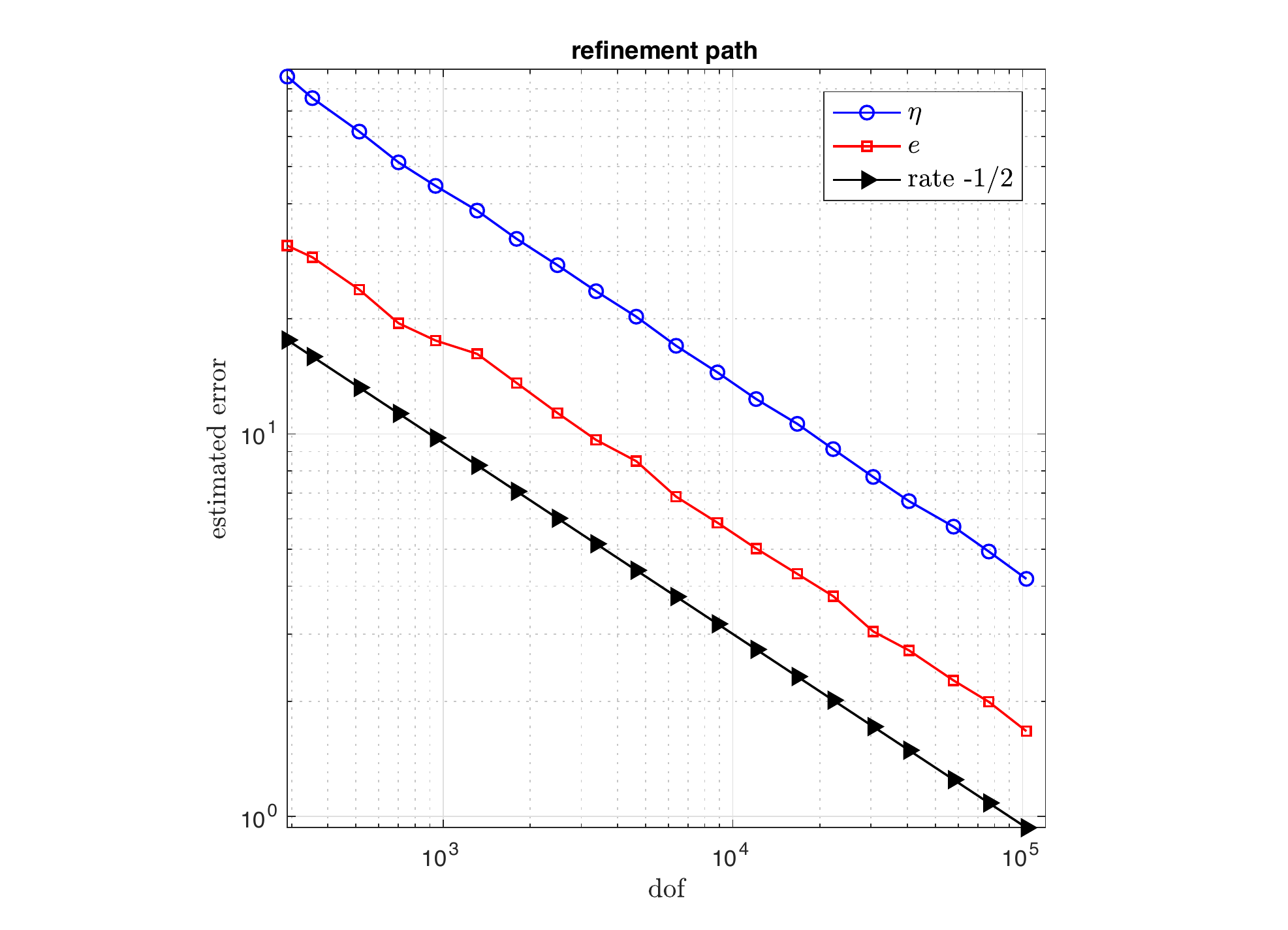}
\label{ex1mu100nu49999res}
\centering
\includegraphics[width=.41\textwidth]{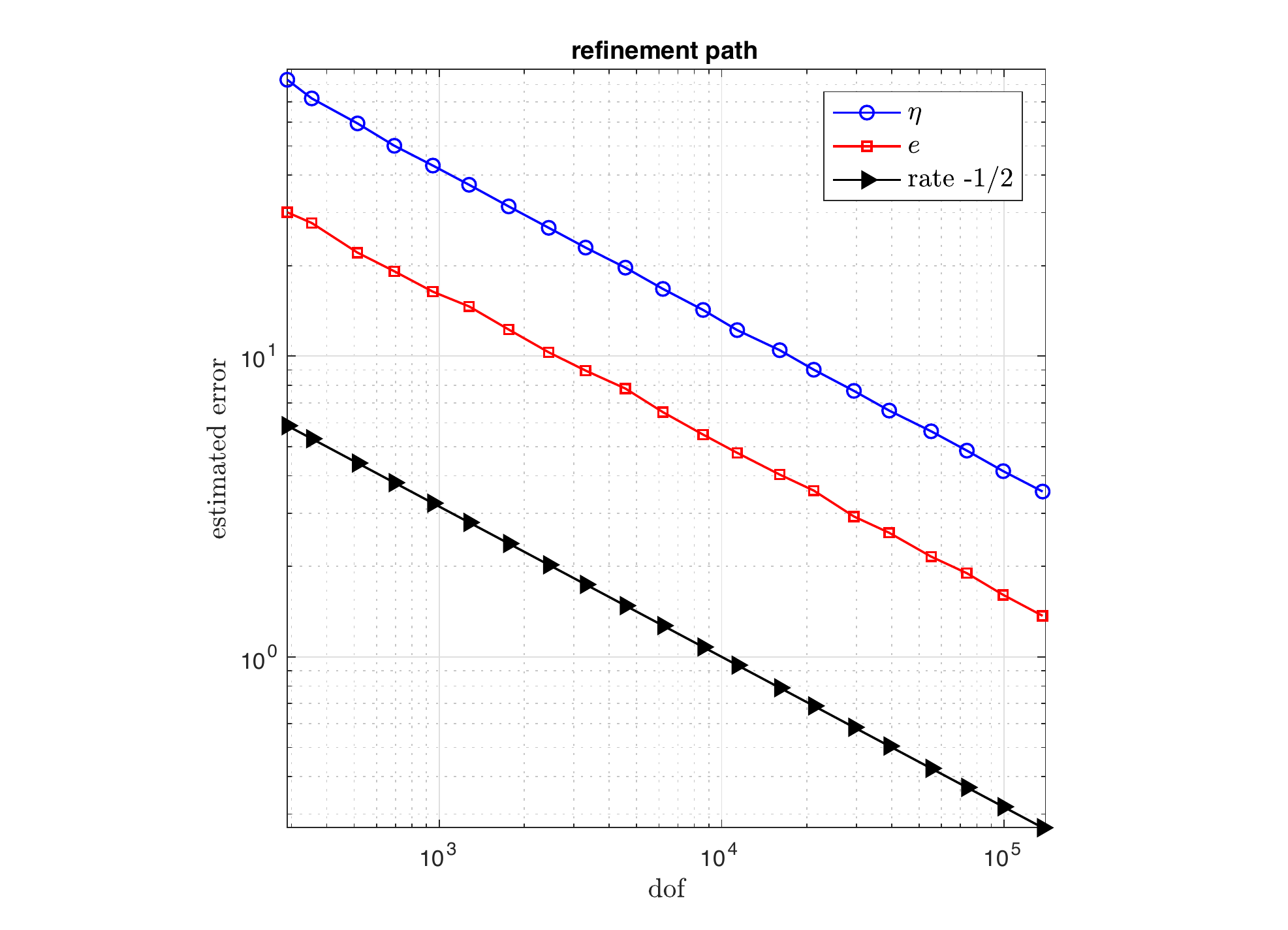}
\label{ex1mu100nu4reseff}
\centering
\includegraphics[width=.41\textwidth]{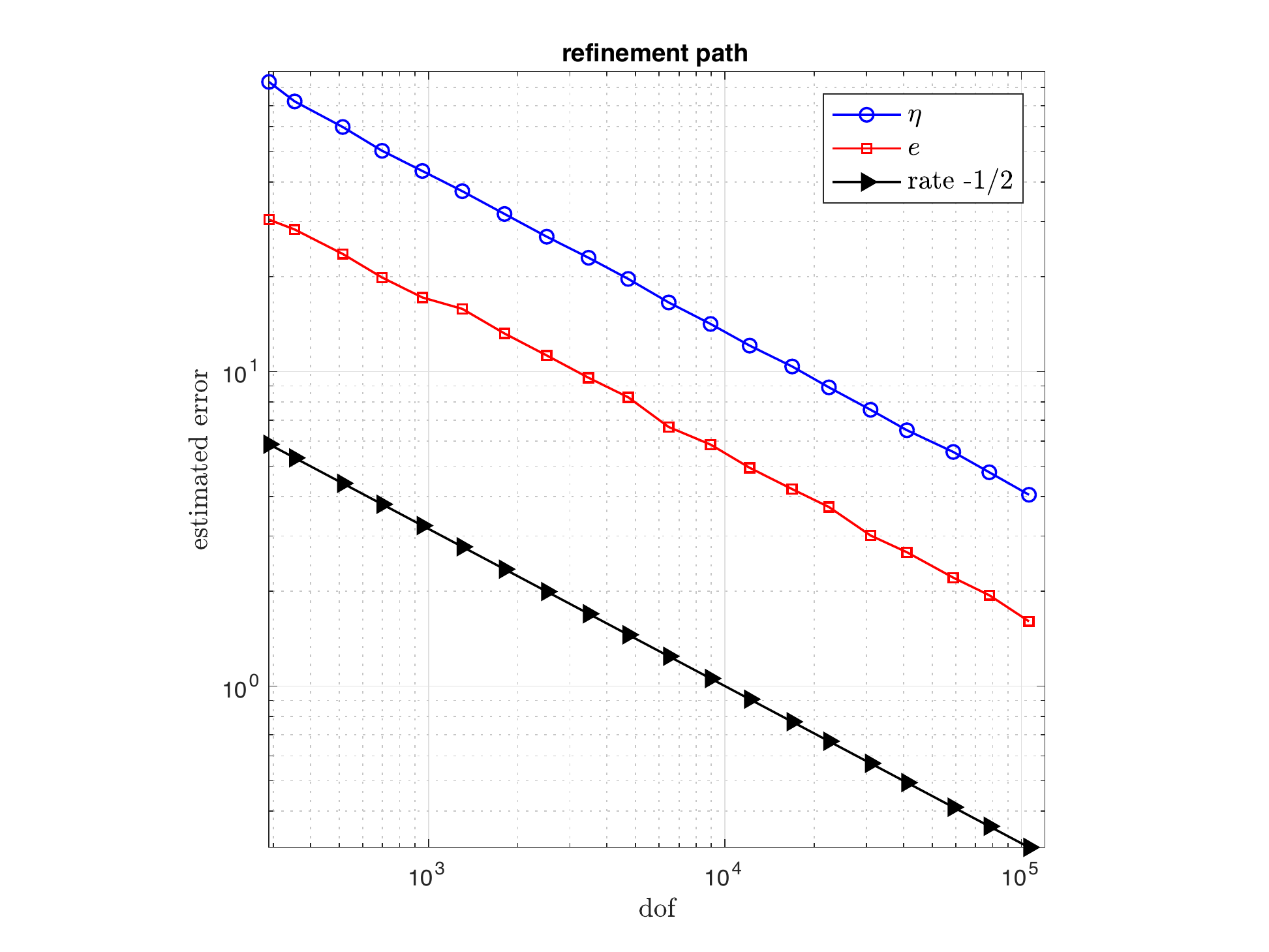}
\label{ex1mu100nu49999reseff}
\caption{
Exact ($e$) and estimated (using the residual estimator $\eta$) energy errors computed using adaptive meshes, \rbl{for} 
 Herrmann (top) and Hydrostatic (bottom) formulations of test problem 1,  
with $\mu=100$  and $\nu=0.4$ (left);    $\mu=100$  and $\nu=0.49999$  (right).}
\label{Ex1mu100res}
\end{figure}


\subsection{\rbl{A nonsmooth solution}}
The second \rbl{test problem is taken from} Ref.~\refcite{wihler2004locking}. 
Again, $\Omega=(0,1)\times (0,1)$ but now $\bm{f}=\bm{0}$ and we impose the 
condition  $\bm{u}=(g,0)^{\top}$ on $\partial \Omega$, where
\begin{align*}
 g=\Big\{\begin{array}{lc}
(1-4(x-\frac{1}{2})^2)^{\frac{1}{2}+\alpha}, &\mbox{on }\; [0,1] \times \{1\},\\
 0,& \hbox{elsewhere on } \partial \Omega.
\end{array} 
\end{align*}
If $\alpha\in(0,\frac{1}{2})$, then the displacement exhibits $H^{\frac{3}{2}+\alpha}$--regularity. Specifically, there are 
singularities at the {top two corners of the domain}. We set the specific value $\alpha=0.1$ so that $\bm{u} \in \rbl{\bm H^{1.6}}(\Omega)$. 
This lack of smoothness is reflected in the convergence behaviour of the estimated energy error. 
Results obtained with the Poisson estimator $\eta_P$  on \emph{uniformly} refined meshes are shown in Figure~\ref{Ex2nu4uni}. 
 Our results suggest that for both the Herrmann and Hydrostatic formulations, the error converges to zero at \rbl{the anticipated}
 suboptimal   rate ($\rbl{r= 0.3}$). However, when we use \emph{adaptively} refined meshes, for both the Herrmann and Hydrostatic 
 formulations, we recover the optimal convergence rate \rbl{of  $r=1/2$}, \rbl{as shown in Figure \ref{Ex2c2}.} 
\rbl{Starting from an  initial mesh with  $N_{0}=1,090$ degrees of freedom, 
the singular solution behaviour is detected  and strong refinement occurs near the top corners. 
Figure \ref{Ex2c2adamesh} shows the meshes that are generated at the first refinement step where  $N_{\ell} \ge 10^{4}$.}

\begin{figure}
\centering
\includegraphics[width=.47\textwidth]{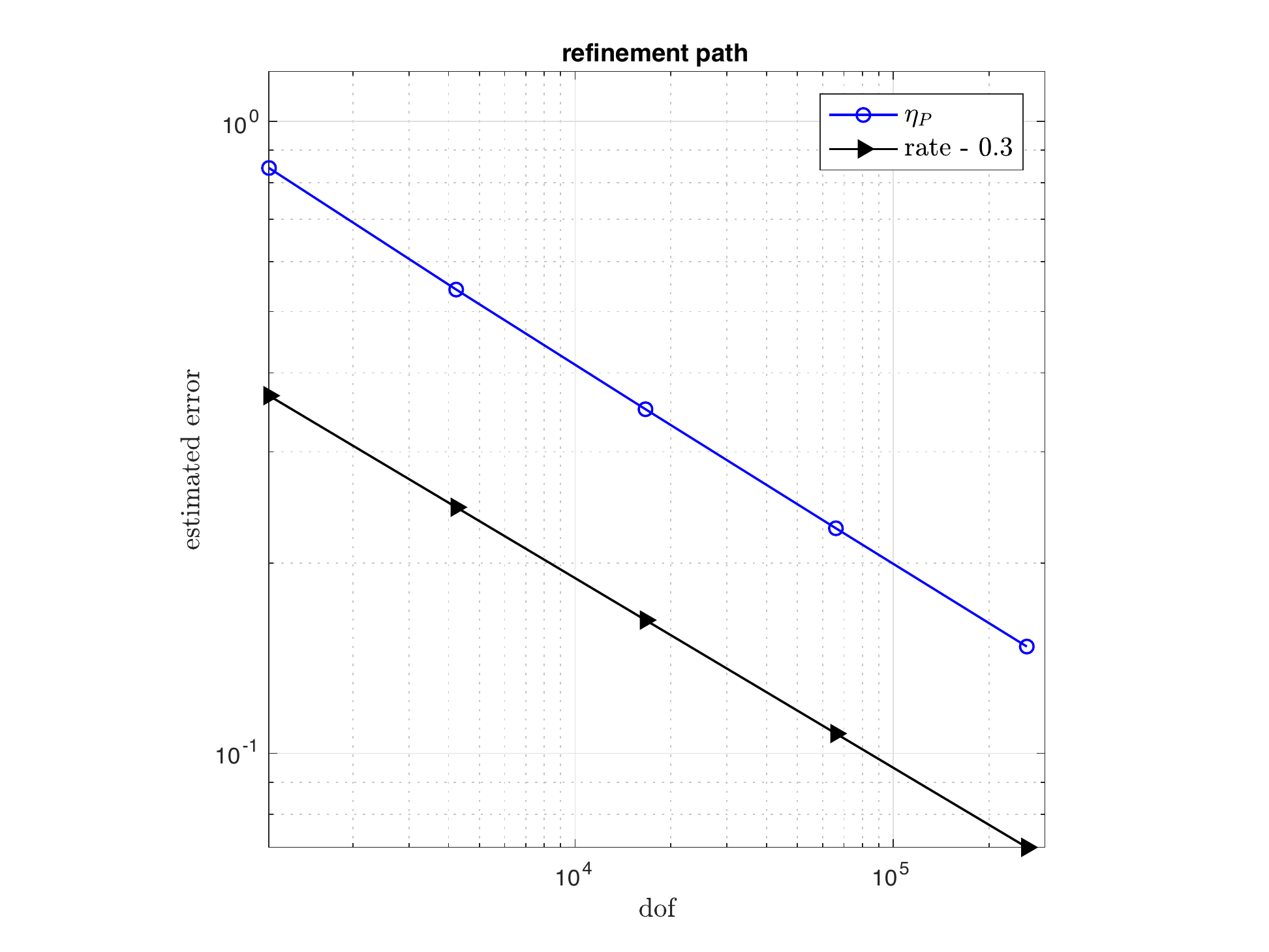}
\label{exc2nu4PEEuni}
\includegraphics[width=.46\textwidth]{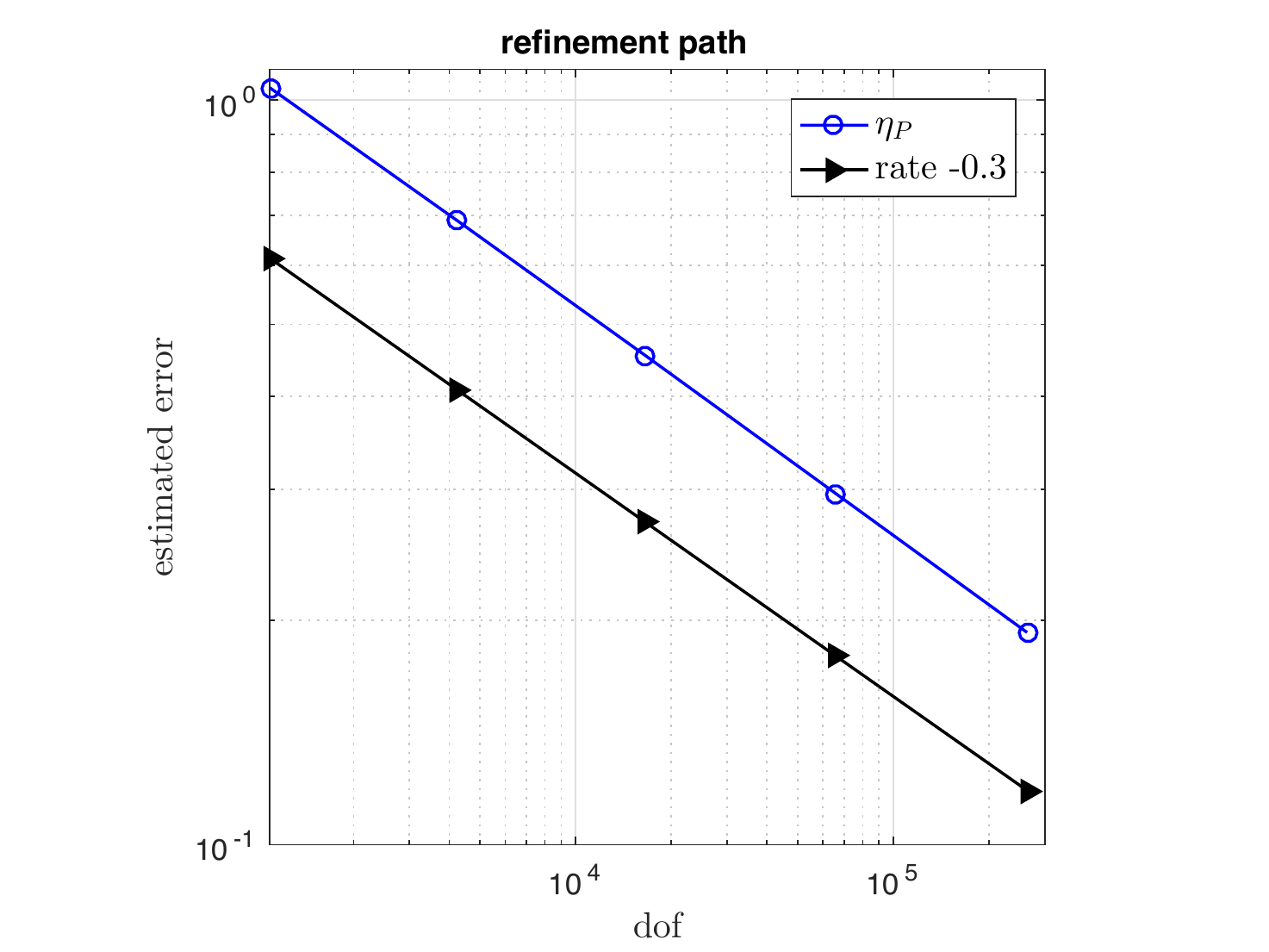}
\label{exc2nu49999PEEuni}
\caption
{Estimated energy errors (using the estimator $\eta_{P}$) computed with \emph{uniform} meshes, \rbl{for}
 Herrmann (left) and Hydrostatic (right) formulations of test problem 2 with  $\mu=1, \nu=0.4$.}
\label{Ex2nu4uni}

\includegraphics[width=.44\textwidth]{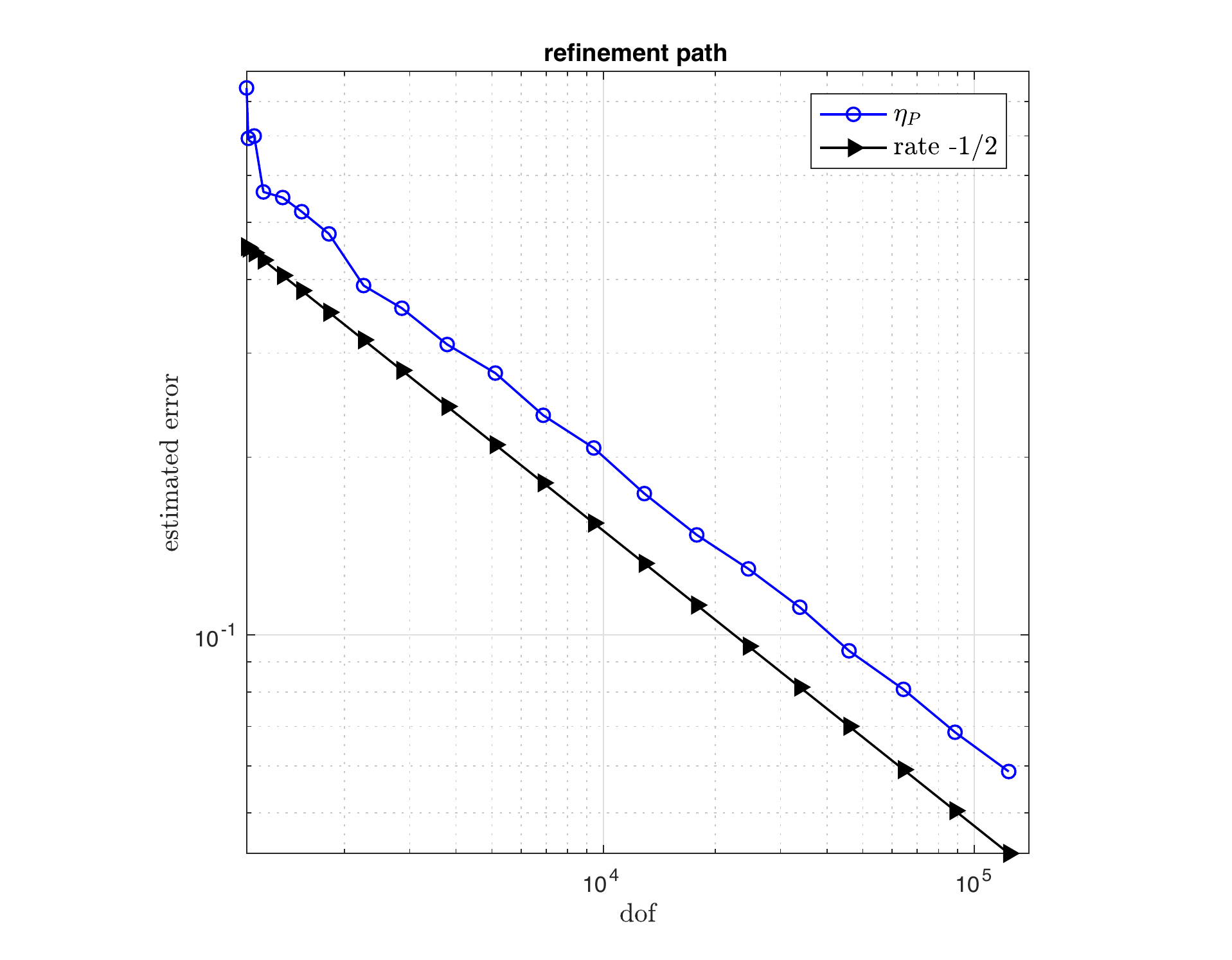}
\includegraphics[width=.43\textwidth]{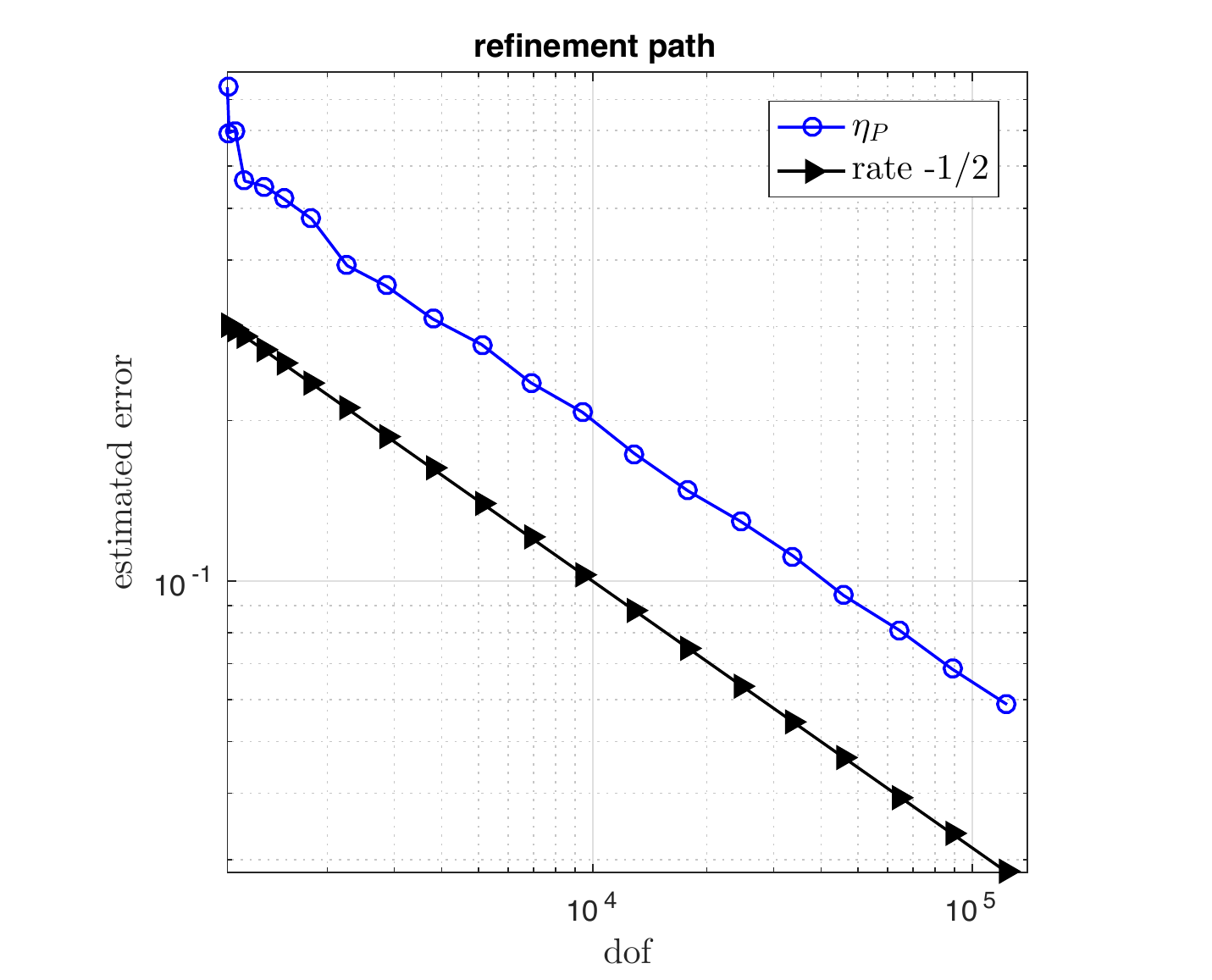}
\centering
\includegraphics[width=.45\textwidth]{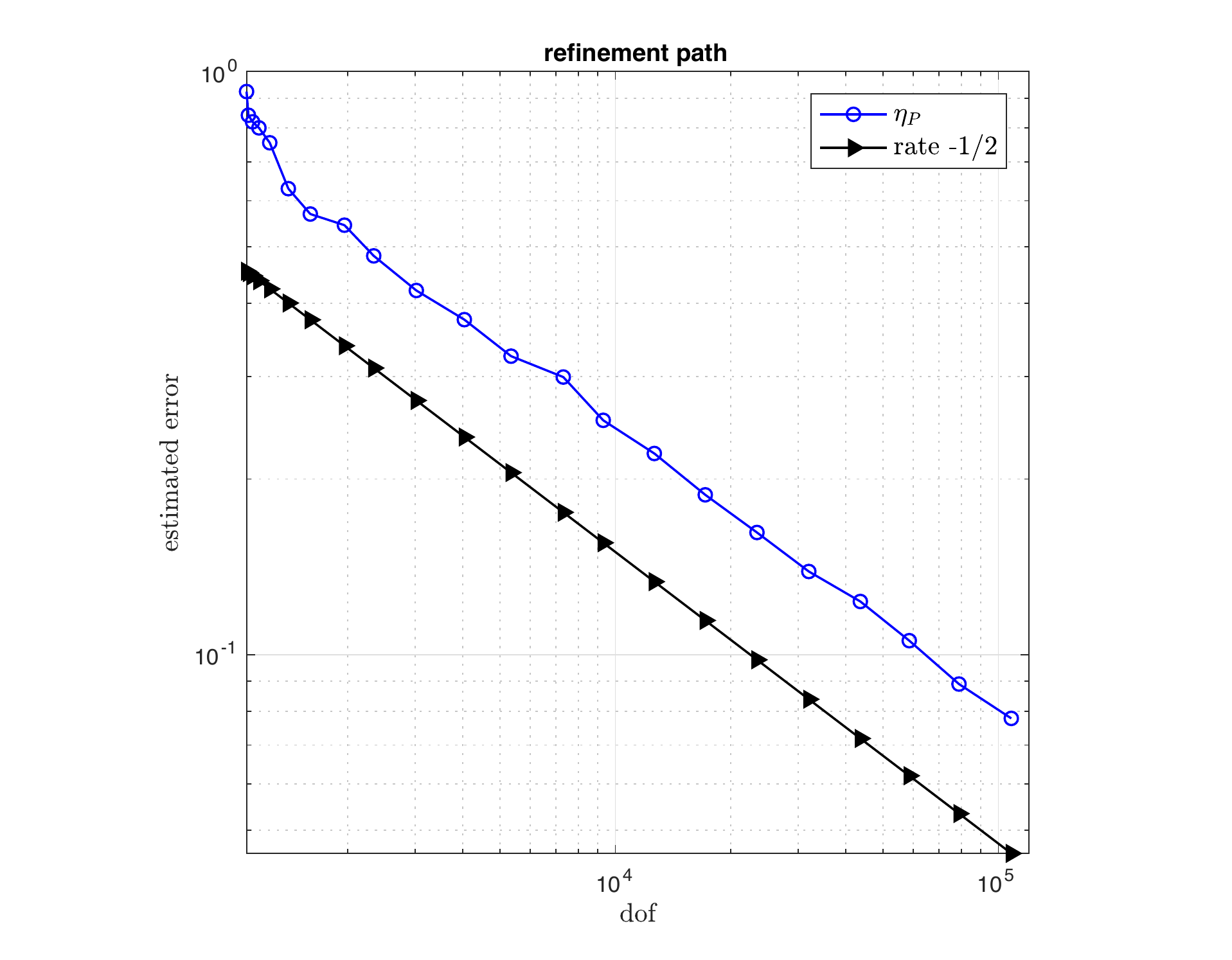}
\includegraphics[width=.45\textwidth]{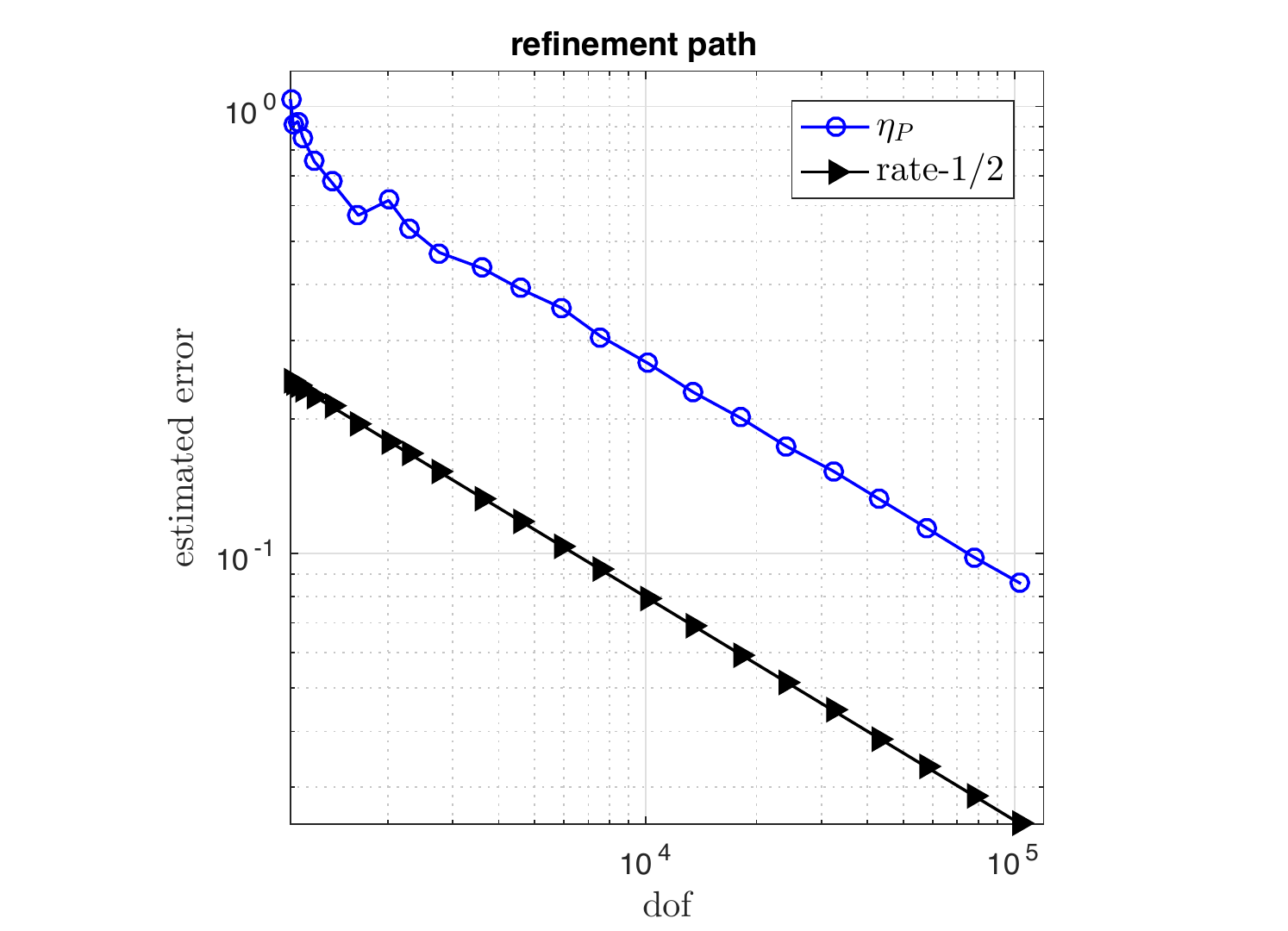}
\caption{Estimated energy errors (using $\eta_{P}$) computed using \emph{adaptive} meshes, \rbl{for}
Herrmann (top) and Hydrostatic (bottom) formulations of test problem 2 with $\mu=1$  and $\nu=0.4$ (left);  $\mu=1$  
and $\nu=0.49999$ (right).}
\label{Ex2c2}
\end{figure}

\begin{figure}[ht!]
\centering
\includegraphics[width=.46\textwidth]{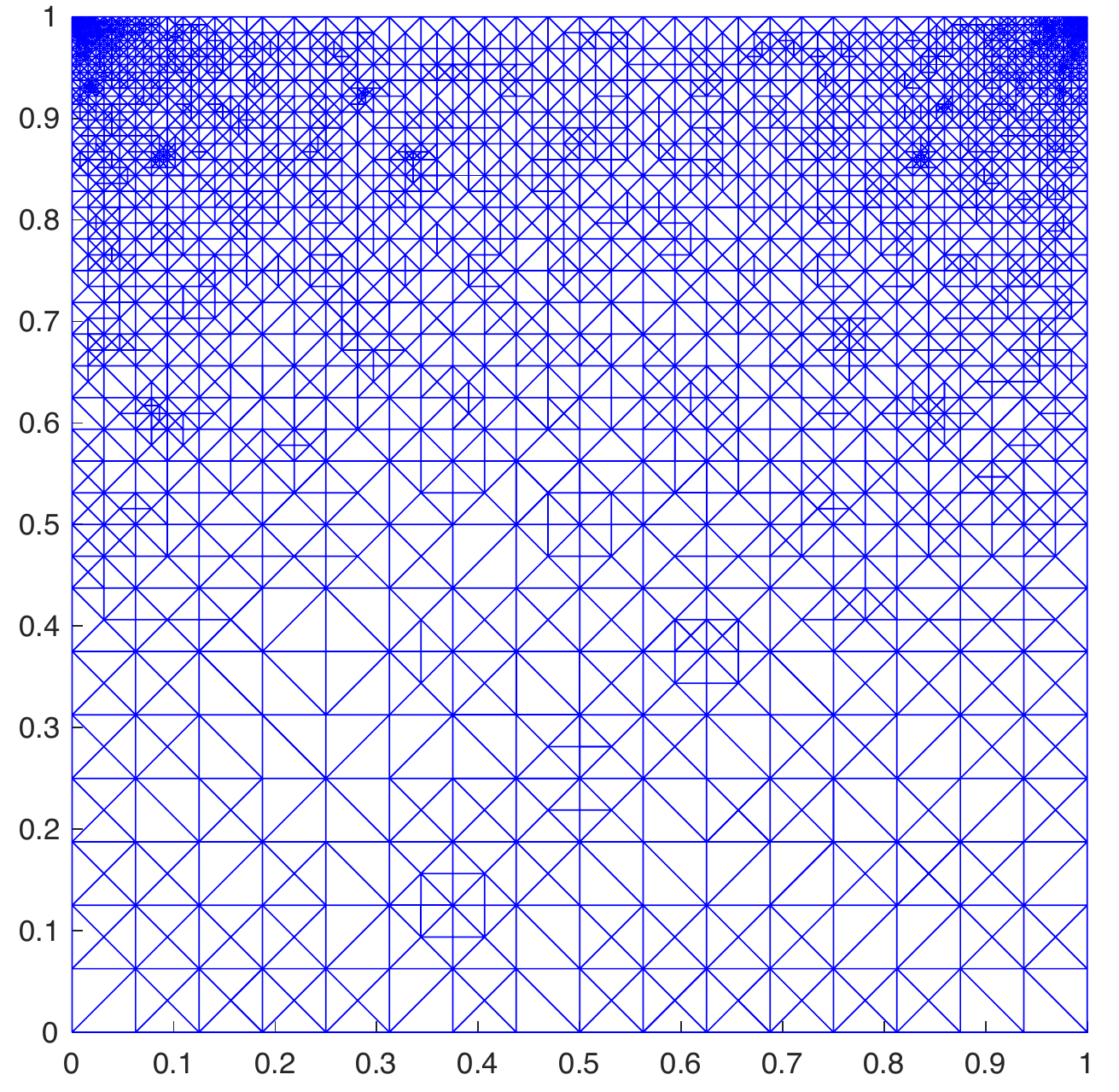}
\label{exc2nu49999ada}
\centering
\includegraphics[width=.45\textwidth]{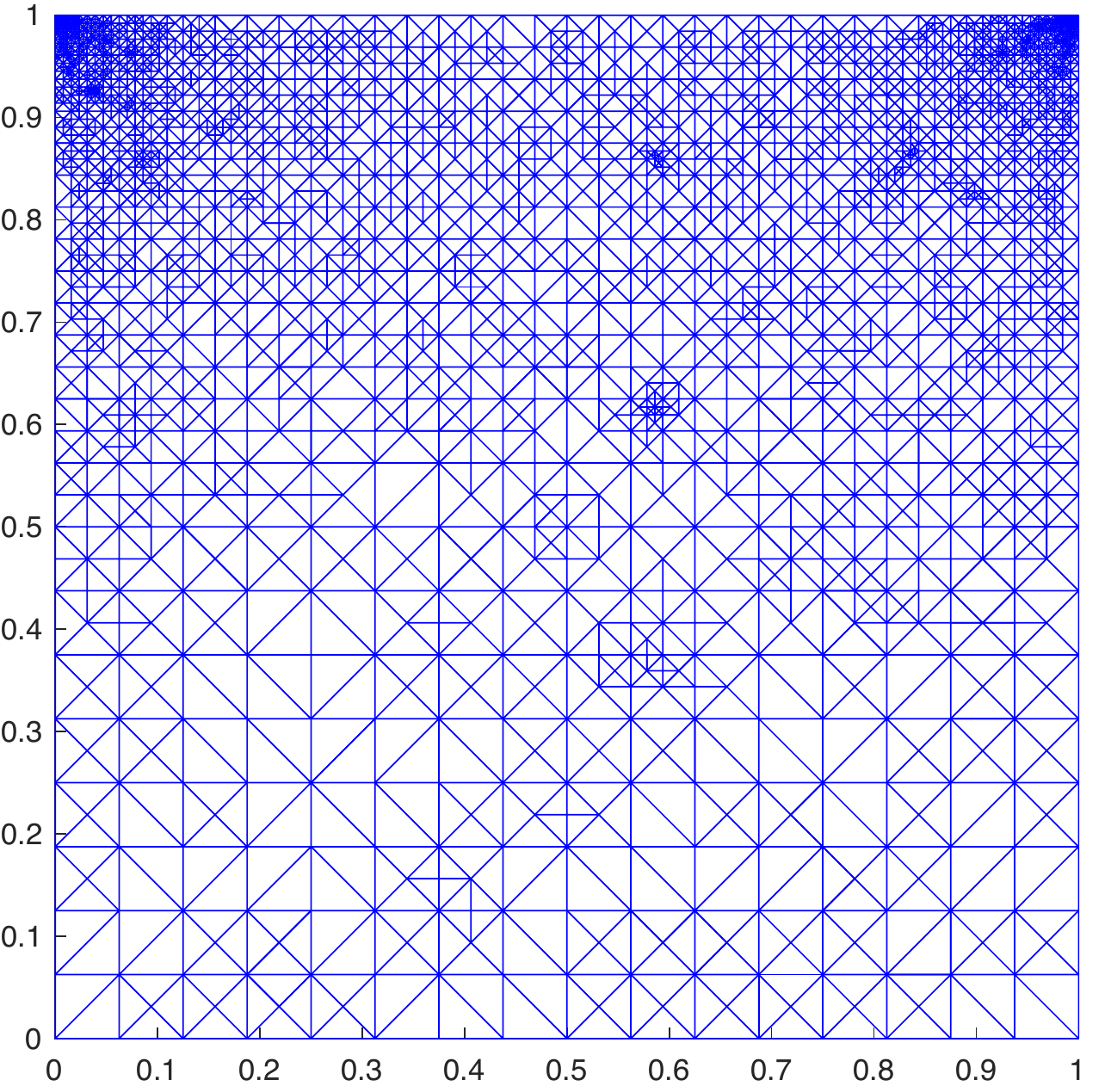}
\label{exc2nu49999hydada}
\caption
{Adaptive meshes generated using $\eta_{P}$ \rbl{for} 
Herrmann (left) and Hydrostatic (right) formulations of test problem 2 with $\nu=0.49999$.
 $N_{\ell}=12,762$ (left), and $N_{\ell}=11,846$ (right).}
\label{Ex2c2adamesh}
\end{figure}

\subsection{\rbl{A singular solution}}
\rbl{To conclude, we discuss a test problem that is considered} in Refs.~\refcite{CJ} and \refcite{wihler2004locking}. 
\rbl{The problem is posed in}  an L-shaped domain $\Omega =(-1,1)\times (-1,1)\setminus (-1,0]\times(-1,0]$. In polar coordinates, the exact displacement is 
\begin{align*}
\bm{u} = \frac{r^\alpha}{2\mu} \left( \begin{array}{c} -(\alpha+1)\cos((\alpha+1)\phi)+(C_2-\alpha+1)C_1\cos((\alpha-1)\phi) \\
(\alpha+1)\sin((\alpha+1)\phi)+(C_2+\alpha-1)C_1\sin((\alpha-1)\phi) \end{array} \right),
\end{align*}
 where $\alpha=0.544483736782$ is a positive solution of $\alpha\sin(\omega)+\sin(2\omega\alpha)=0$ with 
 \begin{align*}
\omega=\frac{3\pi}{4}, \quad  C_1=-\frac{\cos((\alpha+1)\omega)}{\cos((\alpha-1)\omega)},\quad C_2=\frac{2(\lambda+2\mu)}{\lambda+\mu}.
 \end{align*}
The body force is $\bm{f}=\bm{0}$ and the \rbl{nonzero} essential boundary data $\bm{g}$ \rbl{is represented by the piecewise linear interpolant of the given solution}. \rblx{Note that $\int_{\partial \Omega}  \bm{g} \cdot  \bm{n} \, ds \neq 0$.}
To compute the Lam\'{e} constants $\lambda$ and $\mu$, we choose $E=10^5$ and \rbl{set $\nu=0.4$ or $0.49999$}. 
Note that the exact displacement $\bm{u}$ is analytic inside the domain $\Omega$ but $\nabla \bm{u}$ is singular at the origin, so $\bm{u}\notin \bm{H}^2(\Omega)$.  This lack of smoothness is  reflected in the convergence behaviour of the estimated energy error. Results computed with the Poisson estimator $\eta_P$ on \emph{uniform} meshes are shown in Figure~\ref{LshapeE105uni}.  
As \rbl{in the second test problem, we observe the estimated errors converge at a suboptimal rate (here
$r\approx 0.27$)}.
{Moreover, when we use \emph{adaptively} refined meshes, we recover the optimal rate of convergence of $r=0.5$.
This is shown in Figure~\ref{LshapeE105ada}.} 
The singular solution behaviour is detected and strong refinement is \rbl{generated around the re-entrant corner}.
 {While the effectivity indices in Figure~\ref{LshapeE105ada} 
 are not  quite as impressive as those in Figure~\ref{Ex1mu100pes} they remain close to unity 
 (approximately 1.35 when  $\nu=0.4$ and  1.6  when $\nu=0.49999$).}
{We infer from these results  that  $\eta_P$ provides \rblx{an} efficient  and reliable error estimate for both 
Herrmann and Hydrostatic  formulations.}

\begin{figure}
\centering
\includegraphics[width=.4\textwidth]{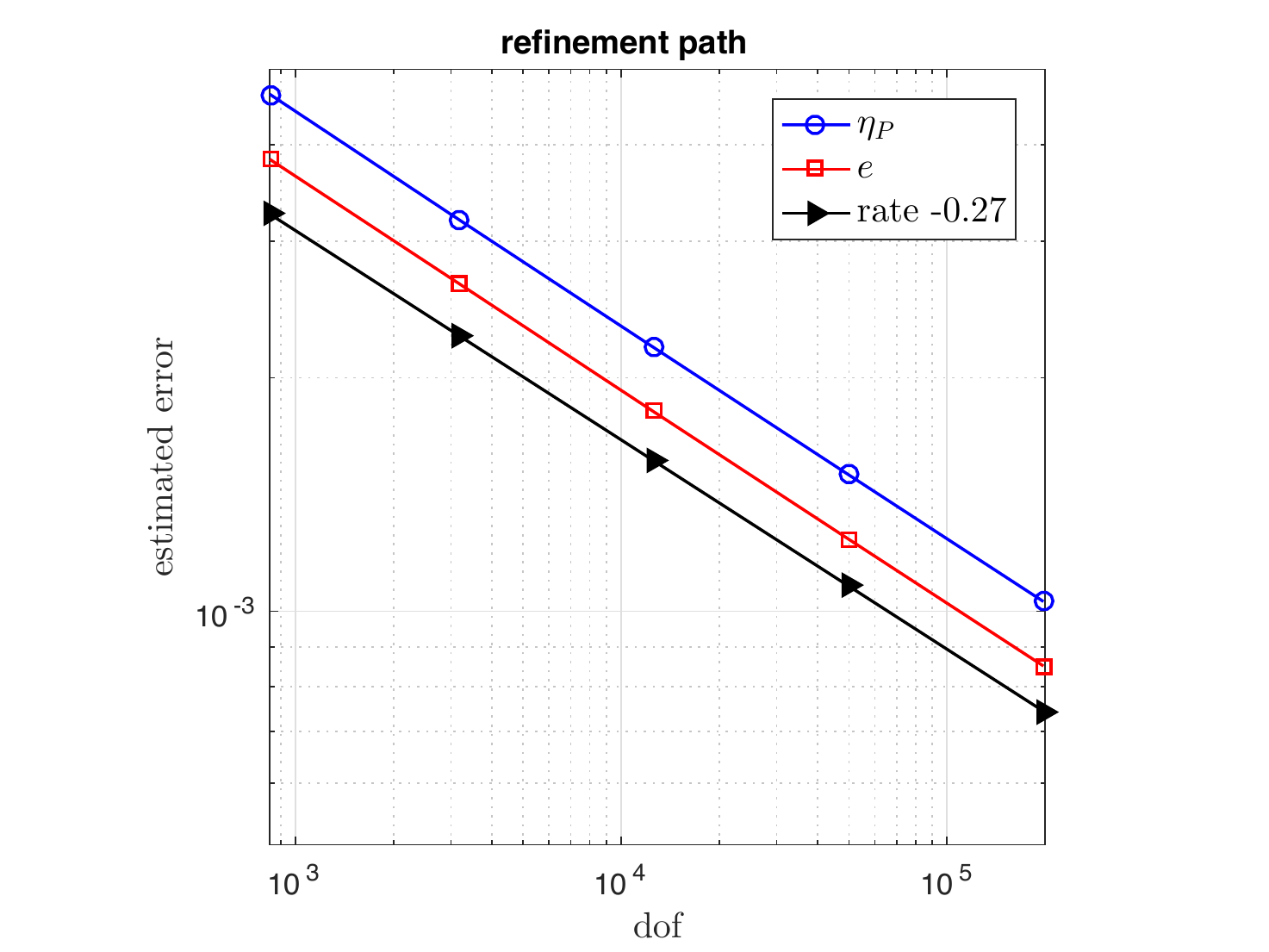}
\includegraphics[width=.4\textwidth]{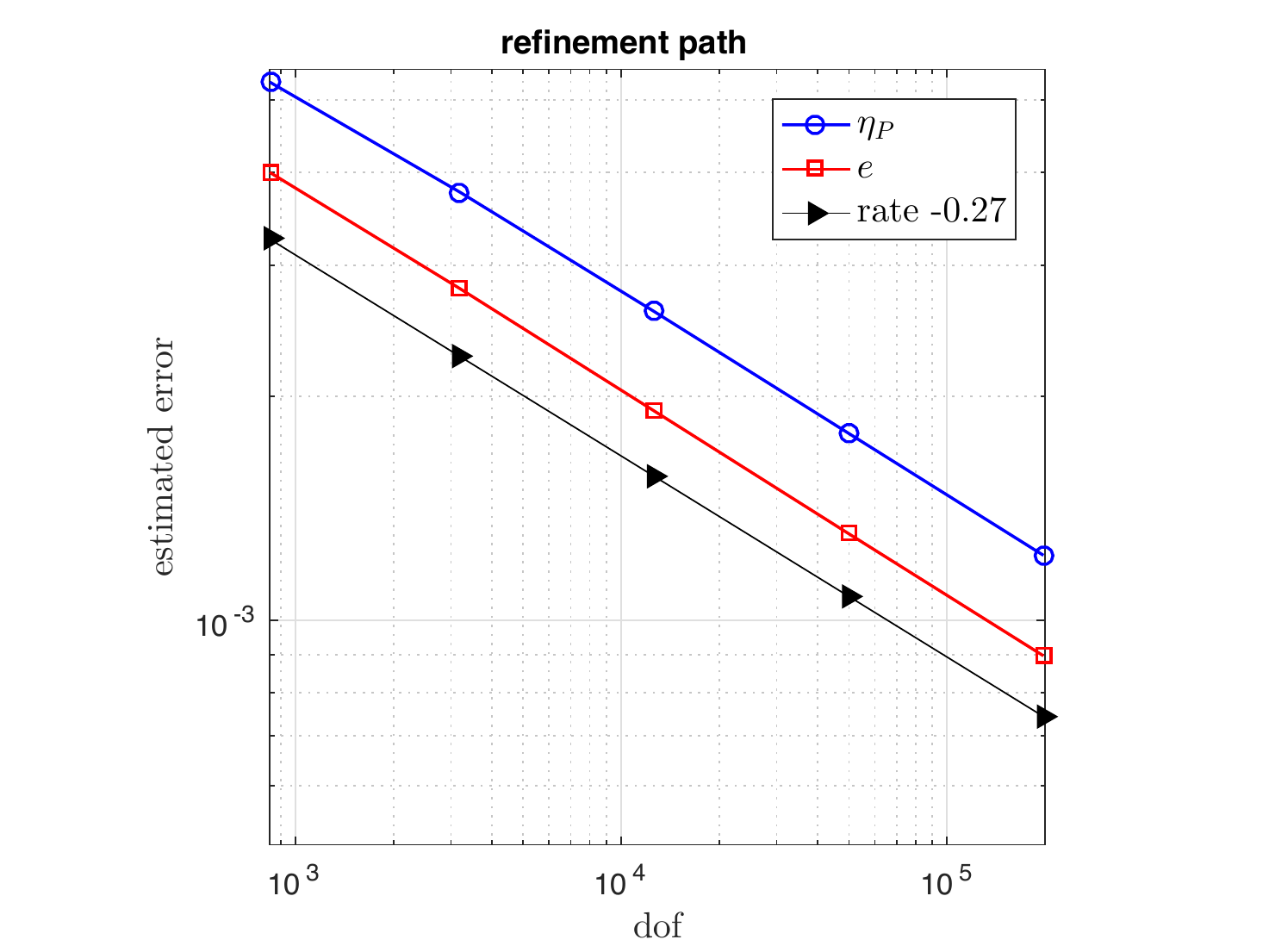}
\includegraphics[width=.4\textwidth]{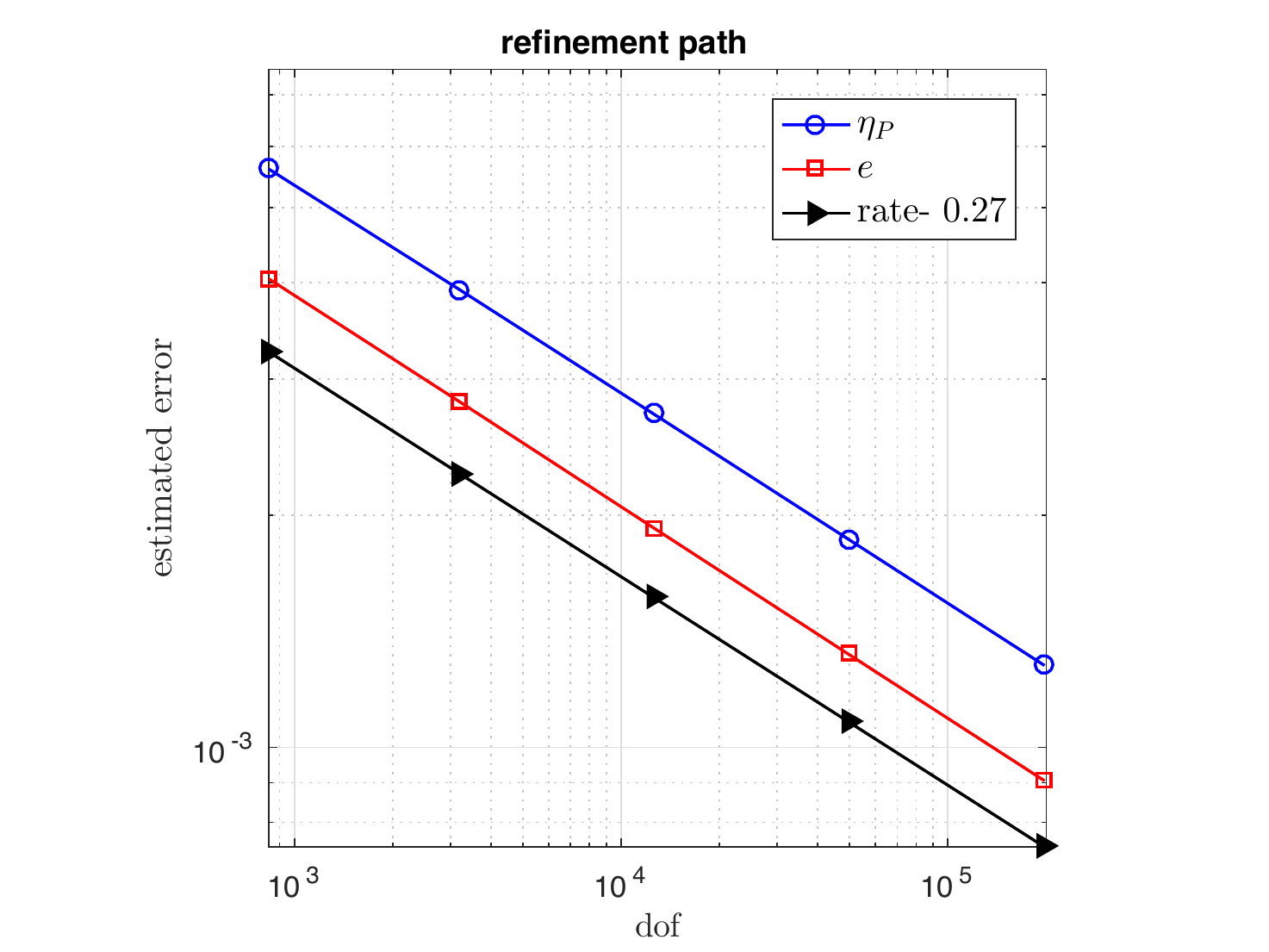}
\includegraphics[width=.4\textwidth]{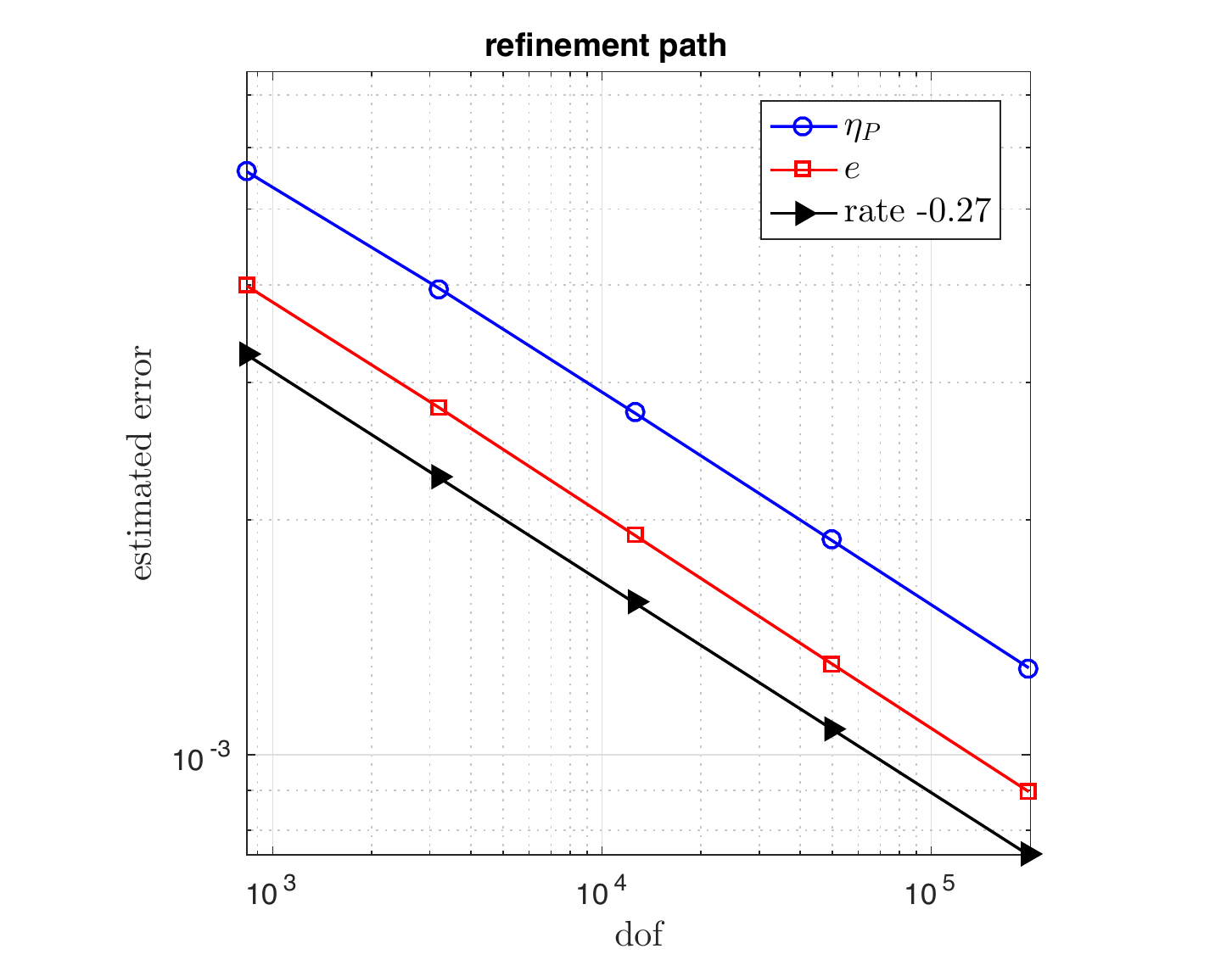}
\label{LshapeE105nu49999hydro}
\caption
{Estimated energy errors (using $\eta_{P}$) computed with \emph{uniform} meshes \rbl{for}
 Herrmann (top) and Hydrostatic (bottom) formulations of test problem 3 with $\nu=0.4$ (left);  $\nu=0.49999$ (right).}
\label{LshapeE105uni}

\includegraphics[width=.42\textwidth]{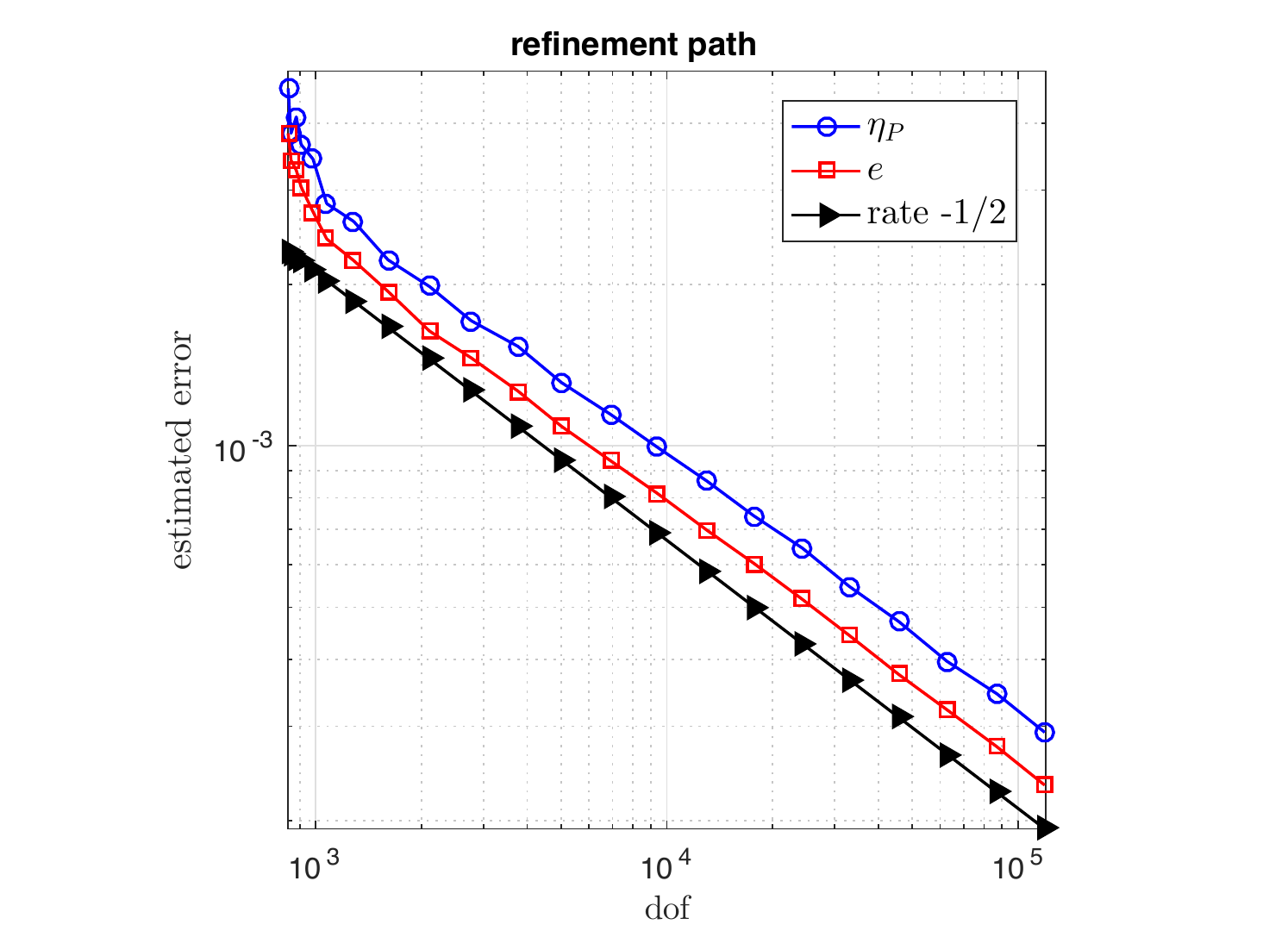}
\includegraphics[width=.42\textwidth]{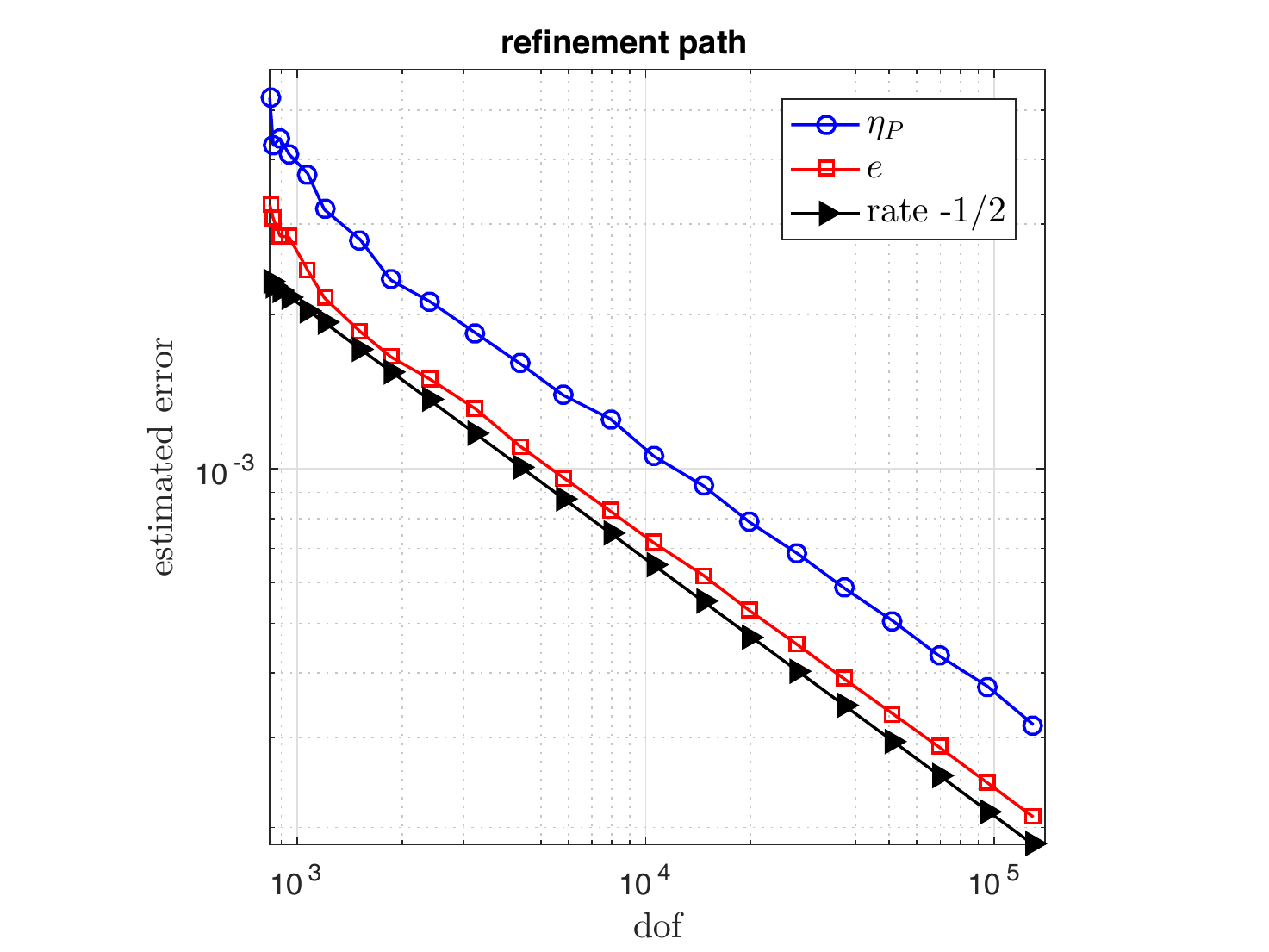}
\includegraphics[width=.42\textwidth]{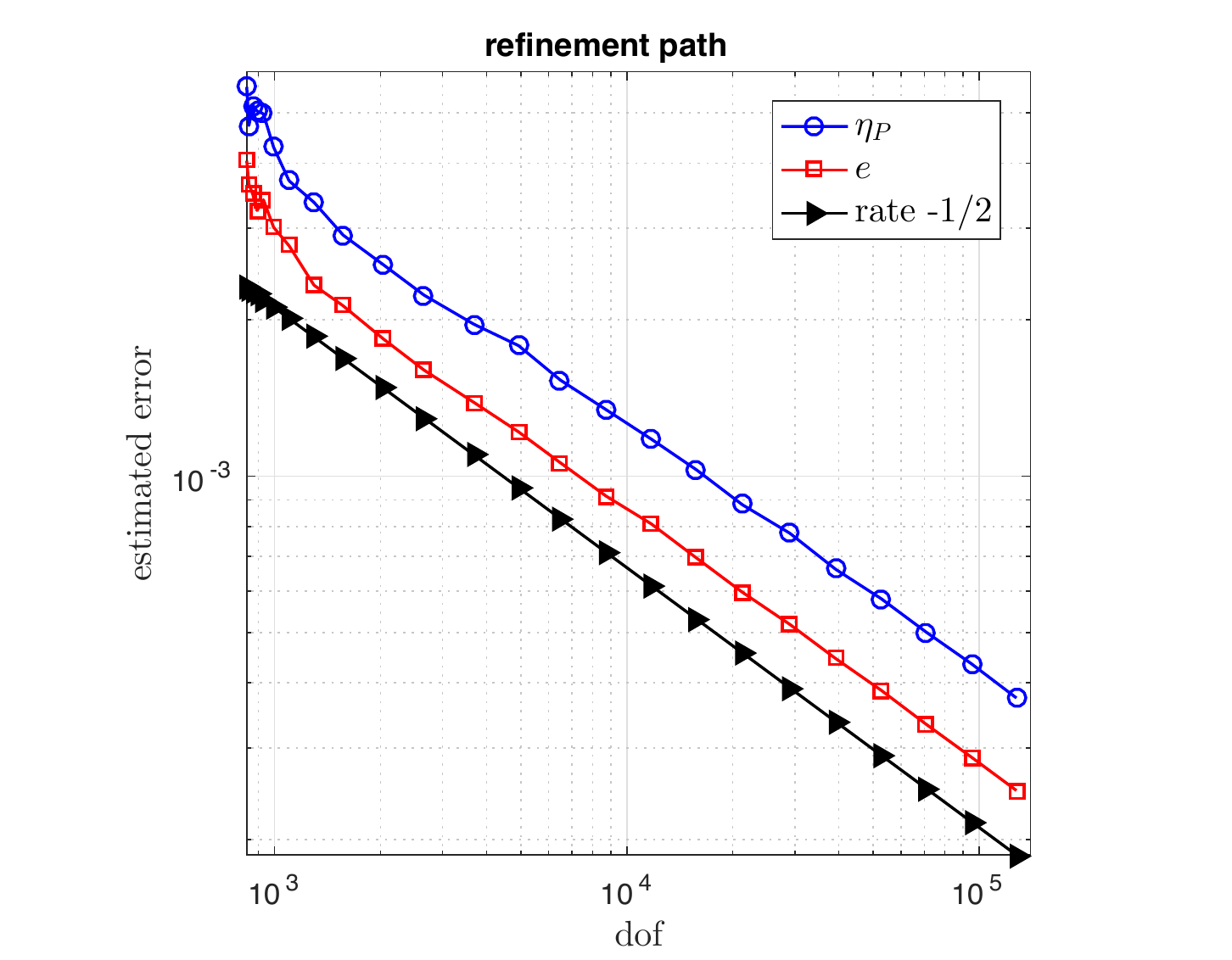}
\includegraphics[width=.42\textwidth]{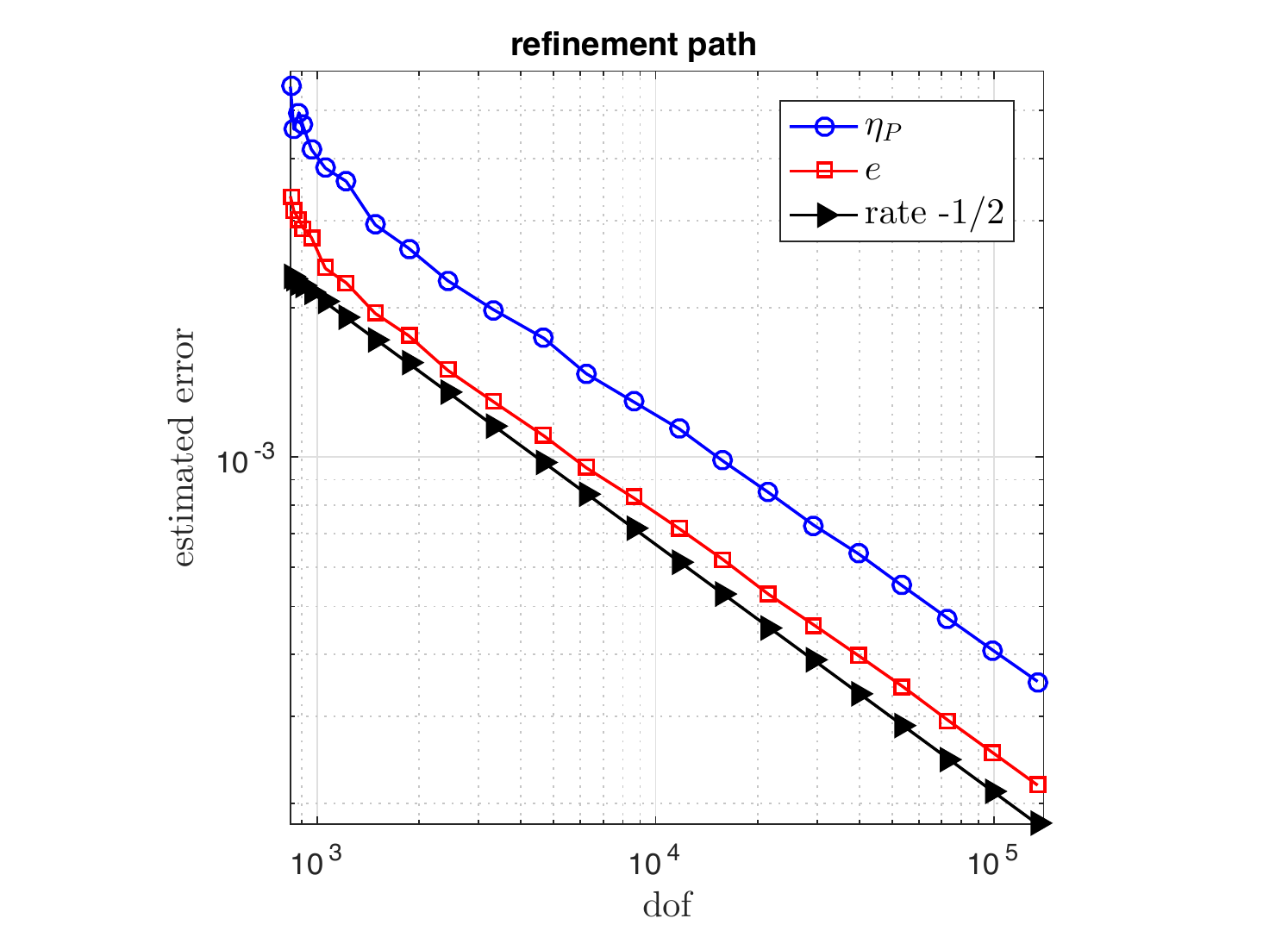}
\caption
{Estimated energy errors (using $\eta_{P}$) computed with \emph{adaptive} meshes \rbl{for}
 Herrmann (top) and Hydrostatic (bottom) formulations of test problem 3 with $\nu=0.4$ (left);  $\nu=0.49999$ (right).}
\label{LshapeE105ada}
\end{figure}



\section{\rbl{Concluding remarks}}\label{conclusions}
\rbl{There are two important contributions in this paper.
First, we have developed a low-order mixed finite element method 
for computing locking-free approximations of linear elasticity problems. The method is computationally cheap 
and challenges the conventional wisdom that it is necessary to start from an inf-sup stable pair of 
finite element spaces. The stabilisation term is weighted by the  problem specific factor {of $1/2\mu$} 
but is otherwise parameter-free.   {Our} a priori error analysis shows that the method provides a 
robust approximation of the energy error. 
That is, the constants in the error bounds do not depend on the Lam\'e coefficients. 
Second, we have described a practical error estimation strategy---based on
solving uncoupled Poisson problems for each displacement component---that
give effectivity indices that are close to unity in all cases that have been tested.
Ensuring robustness in the error estimation process is fundamentally important
when solving problems with large variability in the measurement of material
parameters. Extending this work  to enable the adaptive solution of elasticity problems
with {\it uncertain} material parameters is  the subject of ongoing research.}


\bibliographystyle{plain}
\bibliography{kps}

\end{document}